\documentclass[11pt,a4paper,twoside,fleqn]{amsart}

\usepackage{etoolbox}
\makeatletter
\patchcmd{\ps@headings}{\normalfont}{\sffamily\bfseries}{}{}
\patchcmd{\@secnumfont}{\mdseries}{}{}{}
\patchcmd{\maketitle}{\uppercasenonmath\shorttitle}{}{}{}
\patchcmd{\maketitle}{\MakeUppercase}{}{}{}
\patchcmd{\@maketitle}{\footnotesize\itshape\@dedicatory}{\sffamily\footnotesize\itshape\@dedicatory}{}{}
\patchcmd{\@settitle}{\uppercasenonmath\@title}{\Large\sffamily}{}{}
\patchcmd{\@setauthors}{\MakeUppercase{\authors}}{\large\sffamily\authors}{}{}
\patchcmd{\@setaddresses}{\scshape}{\sffamily\scshape}{}{}
\patchcmd{\@setaddresses}{\ttfamily}{\sffamily}{}{}
\patchcmd{\abstract}{\scshape}{\sffamily\bfseries}{}{}
\patchcmd{\section}{\normalfont\scshape\centering}{\large\sffamily\bfseries}{}{}
\patchcmd{\subsection}{\normalfont\bfseries}{\sffamily\bfseries}{}{}
\patchcmd{\subsubsection}{\normalfont\itshape}{\sffamily\itshape}{}{}
\patchcmd{\paragraph}{\normalfont}{\sffamily\itshape}{}{}
\patchcmd{\subparagraph}{\normalfont}{\sffamily}{}{}
\def\@captionheadfont{\sffamily\bfseries}
\makeatother

\usepackage{hyphenat}
\usepackage[british]{babel}

\usepackage{microtype}

\usepackage{ETbb}
\usepackage[nofontspec,ss]{libertine}

\usepackage[noBBpl]{mathpazo}
\usepackage[scale=1.03]{tx-ds}

\newtheoremstyle{sfplain}{}{}%
{\itshape}{}%
{\sffamily\bfseries}{.}{.5em}{}%

\newtheoremstyle{sfdefinition}{}{}%
{}{}%
{\sffamily\bfseries}{.}{.5em}{}%

\newtheoremstyle{sfremark}{}{}%
{}{}%
{\sffamily\bfseries}{.}{.5em}{}%

\theoremstyle{sfplain}
\newtheorem{theorem}{Theorem}
\newtheorem{lemma}[theorem]{Lemma}
\newtheorem{proposition}[theorem]{Proposition}
\newtheorem{corollary}[theorem]{Corollary}

\theoremstyle{sfdefinition}

\theoremstyle{sfremark}
\newtheorem{remark}[theorem]{Remark}

\newtheorem{conjecture}[theorem]{Conjecture}

\expandafter\let\expandafter\oldproof\csname\string\proof\endcsname
\let\oldendproof\endproof
\renewenvironment{proof}[1][\proofname]{%
  \oldproof[\sffamily\itshape #1]%
}{\oldendproof}

\usepackage[mathic=true]{mathtools}

\usepackage[shortlabels]{enumitem}

\usepackage{booktabs}

\usepackage{tikz}

\usepackage{pgfplots}
\usepgfplotslibrary{fillbetween}

\usepackage[left]{lineno}
\usepackage{etoolbox} 
\newcommand*\linenomathpatch[1]{%
  \cspreto{#1}{\linenomath}%
  \cspreto{#1*}{\linenomath}%
  \csappto{end#1}{\endlinenomath}%
  \csappto{end#1*}{\endlinenomath}%
}
\linenomathpatch{equation}
\linenomathpatch{gather}
\linenomathpatch{multline}
\linenomathpatch{align}
\linenomathpatch{alignat}
\linenomathpatch{flalign}

\usepackage{xcolor}

\usepackage{hyperref}

\usepackage{cm-math}

\definecolor{cmred}{RGB}{190,64,64}
\definecolor{cmyellow}{RGB}{190,156,64}
\definecolor{cmgreen}{RGB}{125,190,64}
\definecolor{cmcyan}{RGB}{64,190,188}
\definecolor{cmblue}{RGB}{64,97,190}
\definecolor{cmmagenta}{RGB}{126,64,190}
\definecolor{cmpink}{RGB}{186,64,190}
\definecolor{cmblack}{RGB}{47,47,47}
\definecolor{cmgrey}{RGB}{127,127,127}
\definecolor{cmwhite}{RGB}{207,207,207}

\begin{document}

\setlength{\mathindent}{3em}

\title[On Gautschi \& Stirling Identities, Asymptotics and Inequalities for the Pi (or Gamma) Function]{On Gautschi \& Stirling Identities, Asymptotics\\{}and Inequalities for the Pi (or Gamma) Function}

\author[M.~Schmidlin]{Marc~Schmidlin}
\address{Marc~Schmidlin. Universit\"at Basel, Departement Mathematik und Informatik. Spiegelgasse 1, 4051 Basel, Schweiz}
\email{marc.schmidlin@unibas.ch}

\thanks{\textit{Funding.} The work of the author was supported
  by the Swiss National Science Foundation (SNSF)
  through the project ``Data Driven High-Dimensional Approximation - Mathematical Analysis and Computing''
  (grant IZVSZ2\_229568).}

\dedicatory{Dedicated to the memory of George Matthewson Clarke}

\begin{abstract}
  We derive two-sided bounds
  for a class of Stirling-type asymptotic formulas
  for piecewise logarithmic interpolations of the pi function,
  and hence also for the factorials and the gamma functions.
  The bounds are derived by first proving some integral identity
  versions of Gautschi's inequality and a class of Stirling-type asymptotic formulas,
  and then bounding these integrals by asymptotically optimal bounds.
  Additionally, all the proofs given rely only on common elementary arguments
  and connect, generalise and possibly improve various results
  that have been published previously.
  Lastly, we provide numerical comparisons concerning the effectiveness
  and behaviour of the bounds and approximations in a graphical manner,
  which clearly indicate that the bounds are asymptotically optimal.
\end{abstract}



\date{\today}

\maketitle

\section{Introduction}

De Moivre and Stirling introduced the well-known asymptotic formula for the factorials,
often refered to as Stirling's asymptotic formula,
\begin{equation*}
  n! \sim \sqrt{2 \pi} n^{n+\frac{1}{2}} e^{-n}
  \qquad\text{for \(n \to \infty\) with \(n \in \Nbbb\),}
\end{equation*}
in the works \cite{DeMoivre,DeMoivre2013,Stirling} circa three centuries ago.
Note that owing to the historical context, see \cite{Archibald26,Pearson24},
it might be more appropriate to call the formula De Moivre-Stirling's asymptotic formula;
however, for the sake of brevity, we will simply use the well-established naming.
Since then many further results and proofs of the asymptotic formula
or modifications of it have been derived and published;
as a small sample of such literature we point to \cite{Burnside1917,HL99,Liu2007,Mortici,Mortici2010,Robbins,Wang}
and the references therein.

We next choose to recall the definition of the pi function \(\Pi\) as introduced by Gauss using the formula
\begin{equation*}
  \Pi(x) \isdef \Gamma(x+1) \isdef \int_{0}^{\infty} t^x e^{-t} \dif\! t .
\end{equation*}
Then it is well known that the pi function is an interpolatory extension of the factorials defined for all natural
to all non-negative real numbers,
i.e. fulfils
\begin{equation*}
  n! = \Pi(n) \qquad\text{for all \(n \in \Nbbb\)}
\end{equation*}
and that Stirling's asymptotic formula holds over the non-negative real numbers;
that is
\begin{equation*}
  \Pi(x) \sim \sqrt{2 \pi} x^{x+\frac{1}{2}} e^{-x}
  \qquad\text{for \(x \to \infty\) with \(x \in \Rbbb_{\geq 0}\).}
\end{equation*}
We will continue to use this rather opionated choice of Gauss' pi function
rather than Legendre's gamma notation to state results,
as the shift introduced in the gamma notation vis-\`a-vis the factorials
will generally only complicate the formulas
and has been described as being ``void of any rationality'', see \cite{Lanczos64}.
Note that the results in literature, as for example cited earlier,
often directly consider the gamma function for results related to Stirling's asymptotic formula
rather than only considering the factorials
and hence these results are immediately available for the pi function,
and thus also the factorials, by shifting the formulas by one.

Importantly, apart from the asymptotic behaviour,
many results in the literature also cover (two-sided) bounds, i.e. inequalities,
stemming from the asymptotic formula,
or give versions that are identities rather than inequalities or asymptotics.
Indeed, there are arguably so many articles and notes published that it is
difficult if not impossible to have an overview of what already has been published.
Thus proofs and results related to Stirling's asymptotic formula
have been rediscovered and republished, cf.\ e.g.\ \cite{BC18}.

An unrelated result concerning the gamma function is commonly known
as Gautschi's inequality for the gamma function,
\begin{equation*}
  y^{1-s} < \frac{\Gamma(y+1)}{\Gamma(y+s)} < (y+1)^{1-s}
  \qquad\text{for \(y \in \Rbbb_{> 0}\) and \(s \in (0,1)\),}
\end{equation*}
see \cite{Gautschi59}.
By substituting \(y = x+1\) and \(s = \alpha\), using the pi function and reordering terms
this turns into
\begin{equation*}
  1 < \frac{\Pi(x) \cdot (x+1)^{\alpha}}{\Pi(x+\alpha)} < \group[\bigg]{\frac{x+2}{x+1}}^{1-\alpha}
  \quad\text{for \(x \in \Rbbb_{\geq 0}\) and \(\alpha \in (0,1)\),}
\end{equation*}
which demonstrates that Gautschi's inequality may be seen as the multiplicative mismatch between the pi function
and its piecewise logarithmic interpolation,
since
\begin{equation*}
  \Pi(x) \cdot (x+1)^{\alpha} = \group[\big]{\Pi(x)}^{1-\alpha} \group[\big]{\Pi(x+1)}^{\alpha}
\end{equation*}
holds.
However,
Gautschi's inequality is actually preceded by the stronger inequality
\begin{equation*}
  \group[\bigg]{\frac{y}{y+s}}^{1-s} \leq \frac{\Gamma(y+s)}{y^s \Gamma(y)} \leq 1
  \qquad\text{for \(y \in \Rbbb_{> 0}\) and \(s \in [0,1]\),}
\end{equation*}
already established in \cite{Wendel48} by Wendel.
Again, substituting \(y = x+1\) and \(s = \alpha\), using the pi function and reordering terms
yields
\begin{equation*}
  1 \leq \frac{\group[\big]{\Pi(x)}^{1-\alpha} \group[\big]{\Pi(x+1)}^{\alpha}}{\Pi(x+\alpha)} \leq \group[\bigg]{\frac{x+1+\alpha}{x+1}}^{1-\alpha}
\end{equation*}
for \(x \in \Rbbb_{\geq 0}\) and \(\alpha \in [0,1]\).

One therefore has
\begin{equation*}
  \group[\big]{\Pi(x)}^{1-\alpha} \group[\big]{\Pi(x+1)}^{\alpha} \sim \Pi(x+\alpha)
  \qquad\text{for \(x \to \infty\)}
\end{equation*}
and thus also
\begin{equation*}
  \group[\big]{\Pi(x)}^{1-\alpha} \group[\big]{\Pi(x+1)}^{\alpha} \sim \sqrt{2 \pi} (x+\alpha)^{x+\alpha+\frac{1}{2}} e^{-x-\alpha}
  \qquad\text{for \(x \to \infty\)}
\end{equation*}
with \(x \in \Rbbb_{\geq 0}\) and \(\alpha \in [0,1]\).
To the author's best knowledge,
such combined Stirling-Gautschi-type asymptotic formulas
for piecewise logarithmic interpolations of the factorials or the pi or gamma function
have not previously been analysed in detail.
As two-sided bounds for such asymptotic formulas turn out to be useful
when discussing the approximation rates of Gevrey smooth functions,
we derive such bounds in detail here,
with the aim of using the two-sided bounds for the logarithmically interpolated factorials
in future work.
As far as we are aware of them, we will reference previously published results
that relate to our results.
However, given the immense amount of results published concerning the asymptotics of the gamma function,
it seems unavoidable that this will be incomplete.

The article is structured as follows:
After stating some remarks concerning notation and preliminaries,
we then derive integral identities for Gautschi's inequality,
and thus for piecewise logarithmic interpolation of the pi function,
as well as for a class of Stirling approximations of the pi function,
which then combine to give a type of Gautschi-Stirling integral identity.
Following this, we use the integral formulas
to derive various two-sided asymptotically optimal bounds.
Finally, we give numerical comparisons showing the effectivity of these bounds
and state concluding remarks.

\section{Notation \& Preliminaries}

We denote the natural numbers including \(0\) by \(\Nbbb\)
and excluding \(0\) by \(\Nbbb_{> 0}\).
Similarily, we use \(\Rbbb_{\geq 0}\) to denote the non-negative real numbers
and \(\Rbbb_{> 0}\) for the set of positive real numbers.
For a real number \(x \in \Rbbb\) we let
\begin{equation*}
  \groupf{x} \isdef \max \groupb{k \in \Zbbb : k \leq x} \in \Zbbb
\end{equation*}
denote the floor of the number and
\begin{equation*}
  \groupb{x} \isdef x - \groupf{x} \in [0,1)
\end{equation*}
be its fractional part.

Additionally,
we let \(D \isdef \Rbbb_{\geq 0}\) and consider the pi function as \(\Pi \colon D \to \Rbbb_{> 0}\)
and define its piecewise logarithmic interpolation
\(\widehat{\Pi} \colon \widehat{D} \to \Rbbb_{> 0}\)
by letting
\begin{equation*}
  \widehat{\Pi}(x, \alpha)
  \isdef \group[\big]{\Pi(x)}^{1-\alpha} \group[\big]{\Pi(x+1)}^\alpha
  = \Pi(x) (x+1)^\alpha
\end{equation*}
for all \((x, \alpha) \in \widehat{D} \isdef \Rbbb_{\geq 0} \times {[0,1]}\).
The piecewise logarithmically interpolated factorials
are hence now given by
\begin{equation*}
  x \widehat{!}
  \isdef \widehat{\Pi}\group[\big]{\groupf{x}, \groupb{x}}
  = \group[\big]{\groupf{x}!}^{1-\groupb{x}} \group[\Big]{\group[\big]{\groupf{x}+1}!}^{\groupb{x}}
\end{equation*}
for any \(x \in \Rbbb_{\geq 0}\).
Note that with these definitions,
Wendel's version of Gautschi's inequality can now be stated as
\begin{align*}
  1
  \leq \frac{\widehat{\Pi}(x, \alpha)}{\Pi(x+\alpha)}
  \leq \group[\bigg]{\frac{x+1+\alpha}{x+1}}^{1-\alpha}
\end{align*}
with \(x \in \Rbbb_{\geq 0}\) and \(\alpha \in [0,1]\).

Lastly, we define
the parametrised class of Stirling-type asymptotic functions
\(S_d \colon \Rbbb_{> {-d}} \to \Rbbb_{> 0}\) by setting
\begin{equation*}
  S_d(x) \isdef \sqrt{2\pi} e^{-d} \group[\big]{x + d}^{x + \frac{1}{2}} e^{-x} ,
\end{equation*}
where \(d \in \Rbbb\), cf.\ e.g.\ \cite{Mortici}.
Given the domain of \(S_d\) it is useful to introduce the following two domains,
where \(\Pi(x)\) and \(S_d(x)\) or \(\widehat{\Pi}(x, \alpha)\) and \(S_d(x+\alpha)\)
are well-defined,
by letting
\begin{equation*}
  D_d \isdef \groupb{x \in D : x > {-d}}
  \quad\text{and}\quad
  \widehat{D}_d \isdef \groupb[\big]{(x, \alpha) \in \widehat{D} : x+\alpha > {-d}} .
\end{equation*}
Note that with these functions the classical asymptotic formula for the factorials
of Stirling, \cite{DeMoivre,DeMoivre2013,Stirling},
and the variant commonly associated with Burnside,
cf. \cite{Burnside1917}, but which, according to \cite{BC18}, was already known to Stirling,
now take the forms,
\begin{equation*}
  n! = \Pi(n) \sim S_0(n) \quad\text{and}\quad
  n! = \Pi(n) \sim S_{\frac{1}{2}}(n) \qquad\text{for \(n \to \infty\).}
\end{equation*}

\section{Basic Properties}

We first note that the constants we have used in defining \(S_d\)
are a consequence of Stirling's asymptotic formula
for the factorials and the Gamma function.

\begin{lemma}\label{lemma:g-Sd-asymp}
  We have that
  \begin{equation*}
    \Pi(x) \sim S_d(x)
    \quad\text{for \(x \to \infty\) and \(d \in \Rbbb\) bounded, with \(x \in D_d\).}
  \end{equation*}
\end{lemma}
\begin{proof}
  We calculate
  \begin{align*}
    \frac{\Pi(x)}{S_d(x)}
    &= \frac{\Pi(x)}{\sqrt{2\pi} e^{-d} \group[\big]{x + d}^{x + \frac{1}{2}} e^{-x}}
    = \frac{\Pi(x)}{\sqrt{2\pi} x^{x + \frac{1}{2}} e^{-x}} e^d \group[\bigg]{\frac{x}{x+d}}^{x + \frac{1}{2}} \\
    &= \frac{\Pi(x)}{\sqrt{2\pi} x^{x + \frac{1}{2}} e^{-x}} \group[\bigg]{1 - \frac{d}{x+d}}^{\frac{1}{2} - d} e^d \group[\bigg]{1 - \frac{d}{x+d}}^{x+d} .
  \end{align*}
  Noting that
  \begin{equation*}
    e^d \group[\bigg]{1 - \frac{d}{x+d}}^{x+d} \sim 1 \qquad\text{for \(x \to \infty\) and \(d \in \Rbbb\) bounded}
  \end{equation*}
  and
  \begin{equation*}
    \group[\bigg]{1 - \frac{d}{x+d}}^{\frac{1}{2}-d} \sim 1 \qquad\text{for \(x \to \infty\) and \(d \in \Rbbb\) bounded}
  \end{equation*}
  hold, the assertion thus follows.
\end{proof}

Next we note that Wendel's or Gautschi's inequality
directly implies the following corollary:
\begin{corollary}\label{corollary:g-interpol-g-asymp}
  We have that
  \begin{equation*}
    \widehat{\Pi}(x, \alpha) \sim \Pi(x+\alpha)
    \quad\text{for \(x \to \infty\) with \((x, \alpha) \in \Rbbb_{\geq 0} \times {[0,1]}\).}
  \end{equation*}
\end{corollary}

Therefore using Lemma~\ref{lemma:g-Sd-asymp}
and Corollary~\ref{corollary:g-interpol-g-asymp}
we also have
\begin{corollary}\label{corollary:g-interpol-Sd-asymp}
  We have that
  \begin{equation*}
    \widehat{\Pi}(x, \alpha) \sim S_{d}(x+\alpha)
    \quad\text{for \(x \to \infty\) and \(d \in \Rbbb\) bounded, with \((x, \alpha) \in \widehat{D}_d\).}
  \end{equation*}
\end{corollary}

Thus it is, for example, obvious that the piecewise logarithmically interpolated factorials
have the following asymptotic behaviour:
\begin{equation*}
  x\widehat{!} \sim S_{d}(x)
  \quad\text{for \(x \to \infty\) and \(d \in \Rbbb\) bounded, with \(x \in D_d\).}
\end{equation*}

\section{Gautschi's Inequality as an Integral Identity}\label{sec:GautschiIdentity}

We first define the logarithmic multiplicative mismatch between the pi function
and its piecewise logarithmic interpolations by letting \(\widehat{\imath} \colon \widehat{D} \to \Rbbb\)
be defined as
\begin{equation*}
  \widehat{\imath}(x, \alpha) \isdef \log \frac{\widehat{\Pi}(x, \alpha)}{\Pi(x+\alpha)} .
\end{equation*}
With this, we consider differences of the logarithmic multiplicative mismatch \(\widehat{\imath}\).
\begin{lemma}\label{lemma:loginterpol-diff}
  For any \((x, \alpha) \in \widehat{D}\),
  we have
  \begin{equation*}
    \widehat{\imath}(x, \alpha) - \widehat{\imath}(x+1, \alpha)
    = \int_1^2 \frac{\phi_\alpha(t)}{x+t} \dif\!t ,
  \end{equation*}
  where \(\phi_\alpha \colon \Rbbb \to \Rbbb\) is defined by
  \begin{equation*}
    \phi_\alpha(t) \isdef \begin{cases}
      1 - \alpha & \text{if \(\{t\} \leq \alpha\),} \\
      -\alpha & \text{if \(\{t\} > \alpha\).} 
    \end{cases}
  \end{equation*}
\end{lemma}
\begin{proof}
  We begin by calculating
  \begin{align*}
    \widehat{\imath}(x, \alpha) - \widehat{\imath}(x+1, \alpha)
    &= \log \frac{\widehat{\Pi}(x, \alpha) \Pi(x+1+\alpha)}{\Pi(x+\alpha)\widehat{\Pi}(x+1, \alpha)} \\
    &= \log \frac{\Pi(x) (x+1)^\alpha \Pi(x+\alpha) (x+1+\alpha)}{\Pi(x+\alpha) \Pi(x) (x+1) (x+2)^\alpha} \\
    &= \log \frac{(x+1+\alpha)^{1-\alpha} (x+1+\alpha)^\alpha}{(x+1)^{1-\alpha} (x+2)^\alpha} \\
    &= (1-\alpha) \log\frac{x+1+\alpha}{x+1} + \alpha \log\frac{x+1+\alpha}{x+2}
  \end{align*}
  where we have used that
  \begin{gather*}
    \Pi(x+1) = (x+1)\Pi(x) ,
    \quad
    \Pi(x+2) = (x+2)(x+1)\Pi(x) \\
    \text{and}\quad
    \Pi(x+1+\alpha) = (x+1+\alpha)\Pi(x+\alpha)
  \end{gather*}
  hold.
  Next, using
  \begin{equation*}
    \log\frac{x+1+\alpha}{x+1+s}
    = \log(x+1+\alpha) - \log(x+1+s)
    = \int_s^\alpha \frac{1}{x+1+t} \dif\!t
  \end{equation*}
  with \(s \in \{0, 1\}\)
  we arrive at
  \begin{align*}
    \widehat{\imath}(x, \alpha) - \widehat{\imath}(x+1, \alpha)
    &= (1 - \alpha) \int_0^\alpha \frac{1}{x+1+t} \dif\!t
    + \alpha \int_1^\alpha \frac{1}{x+1+t} \dif\!t \\
    &= \int_0^\alpha \frac{1 - \alpha}{x+1+t} \dif\!t
    + \int_\alpha^1 \frac{-\alpha}{x+1+t} \dif\!t \\
    &= \int_0^1 \frac{\phi_\alpha(t)}{x+1+t} \dif\!t .
  \end{align*}
  As \(\phi_\alpha\) is \(1\)-periodic, this
  obviously implies the asserted formula.
\end{proof}

With the formula giving these differences as an integral,
we can also give a formula for the logarithmic multiplicative mismatch \(\widehat{\imath}\)
as an integral.
\begin{lemma}\label{lemma:loginterpol}
  For any \((x, \alpha) \in \widehat{D}\),
  we have
  \begin{equation*}
    \widehat{\imath}(x, \alpha)
    = \int_1^{\infty} \frac{\phi_\alpha(t)}{x+t} \dif\!t ,
  \end{equation*}
  where the integral may either be considered as a Lebesgue integral
  or an improper Riemann integral.
\end{lemma}
\begin{proof}
  For any \(k \in \Nbbb\),
  we have
  \begin{equation*}
    \widehat{\imath}(x, \alpha) - \widehat{\imath}(x+k, \alpha)
    = \int_{1}^{1+k} \frac{\phi_\alpha(t)}{x+t} \dif\!t
  \end{equation*}
  by Lemma~\ref{lemma:loginterpol-diff}.
  Since we know from Lemma~\ref{corollary:g-interpol-g-asymp} that
  \begin{equation*}
    \lim_{k \to \infty} \widehat{\imath}(x+k, \alpha) = 0
  \end{equation*}
  must hold, the assertion follows by letting \(k\) tend to infinity.
\end{proof}

From Lemma~\ref{lemma:loginterpol} 
we arrive directly at the following integral identity version of Gautschi's inequality:
\begin{theorem}\label{theorem:loginterpol}
  For any \((x, \alpha) \in \widehat{D}\),
  we have
  \begin{equation*}
    \frac{\widehat{\Pi}(x, \alpha)}{\Pi(x + \alpha)} = \exp\group[\Bigg]{\int_1^{\infty} \frac{\phi_\alpha(t)}{x+t}  \dif\!t} ,
  \end{equation*}
  or restating the identity in the form commonly used for stating Gautschi's inequality,
  \begin{equation*}
    \frac{\Gamma(y+1)}{\Gamma(y+s)} = y^{1-s} \exp\group[\Bigg]{\int_0^{\infty} \frac{\phi_s(t)}{y+t}  \dif\!t}
  \end{equation*}
  for any \(y \in \Rbbb_{\geq 1}\) and \(s \in {[0,1]}\).
\end{theorem}

\section{Stirling's Asymptotic Formulas as Integral Identities}\label{sec:StirlingIdentity}

In this section we will derive integral identities for the asymptotic formulas.
Especially,
we are deriving a generalisation of the types of integral identities
found in \cite{Liu2007} and \cite{Mortici2010} for the asymptotic formulas for the factorials.

For any \(d \in \Rbbb\), we define the function \(m_d \colon D_d \to \Rbbb\) defined by
\begin{equation*}
  m_d(x) \isdef \log \frac{\Pi(x)}{S_d(x)} .
\end{equation*}
Note that \(m_d(x)\) is thus simply the logarithm of the multiplicative
mismatch between the pi function and the asymptotic formulas.
Rewriting this we have that
\begin{equation*}
  \Pi(x) = S_d(x) \exp\group[\big]{m_d(x)}
\end{equation*}
holds for all \(x \in D_d\) where \(d \in \Rbbb\).
Note that Lemma~\ref{lemma:g-Sd-asymp} directly implies that
\begin{equation*}
  m_d(x) \sim 0
  \quad\text{for \(x \to \infty\) and \(d \in \Rbbb\) bounded, with \(x \in D_d\).}
\end{equation*}

We now first derive the following relation between the logarithmic multiplicative mismatches \(m_d\).
\begin{lemma}\label{lemma:md-to-md}
  For any \(c,d \in \Rbbb\) and \(x \in D_c \cap D_d\),
  we have
  \begin{equation*}
    m_d(x) = m_c(x) + \int_d^c \frac{\tfrac{1}{2} - t}{x + t} \dif\!t .
  \end{equation*}
\end{lemma}
\begin{proof}
  We start by calculating
  \begin{equation*}
    m_d(x) - m_c(x)
    = \log \frac{\Pi(x)}{S_d(x)} - \log \frac{\Pi(x)}{S_c(x)}
    = \log \frac{S_c(x)}{S_d(x)} .
  \end{equation*}
  Evaluating this we arrive at
  \begin{align*}
    \log \frac{S_c(x)}{S_d(x)}
    &= \log \frac{\sqrt{2\pi} e^{-c} \group[\big]{x + c}^{x + \frac{1}{2}} e^{-x}}{\sqrt{2\pi} e^{-d} \group[\big]{x + d}^{x + \frac{1}{2}} e^{-x}} \\
    &= d - c + \group[\big]{x + \tfrac{1}{2}} \group[\big]{\log(x+c) - \log(x+d)} .
  \end{align*}
  Now, using the fact that
  \begin{equation*}
    \log(x+d) = \log(x+c) + \int_c^d \frac{1}{x+t} \dif\!t
  \end{equation*}
  holds,
  we have
  \begin{equation*}
    \log \frac{S_c(x)}{S_d(x)}
    = d - c + \int_d^c \frac{x+\tfrac{1}{2}}{x+t} \dif\!t
    = d - c + \int_d^c \frac{x+t+\tfrac{1}{2}-t}{x+t} \dif\!t
  \end{equation*}
  and this now directly implies
  \begin{equation*}
    m_d(x) - m_c(x)
    = d - c + \int_d^c 1 \dif\!t + \int_d^c \frac{\tfrac{1}{2}-t}{x+t} \dif\!t
    = \int_d^c \frac{\tfrac{1}{2}-t}{x+t} \dif\!t . \qedhere
  \end{equation*}
\end{proof}

Next,
we consider differences of the logarithmic multiplicative mismatches \(m_d\)
when \(d \in {[0,1]}\).
\begin{lemma}\label{lemma:md-diff}
  For any \(d \in {[0,1]}\) and \(x \in D_d\),
  we have
  \begin{equation*}
    m_d(x) - m_d(x+1) = \int_{d}^{1+d} \frac{\frac{1}{2} - \groupb{t}}{x+t} \dif\!t .
  \end{equation*}
\end{lemma}
\begin{proof}
  We begin by calculating
  \begin{align*}
    m_d(x) - m_d(x+1)
    &= \log \frac{\Pi(x) S_d(x+1)}{S_d(x) \Pi(x+1)} \\
    &= \log \frac{\sqrt{2\pi} e^{-d} \group[\big]{x+1 + d}^{x+1 + \frac{1}{2}} e^{-x-1}}{(x+1) \sqrt{2\pi} e^{-d} \group[\big]{x + d}^{x + \frac{1}{2}} e^{-x}} ,
  \end{align*}
  where we have used that \(\Pi(x+1) = (x+1)\Pi(x)\).
  Simplifying this we arrive at
  \begin{multline*}
    m_d(x) - m_d(x+1) \\
    = \group[\big]{x+1 + \tfrac{1}{2}} \log(x+1 + d) - \group[\big]{x + \tfrac{1}{2}} \log(x + d) - \log(x+1) - 1 .
  \end{multline*}
  We now note that we may rewrite the various terms as follows:
  \begin{align*}
    \log(x+1 + d) &= \log(x+d) + \int_d^1 \frac{1}{x+t} \dif\!t + \int_1^{1+d} \frac{1}{x+t} \dif\!t ,\\
    \log(x+1) &= \log(x+d) + \int_d^1 \frac{1}{x+t} \dif\!t  \\
    \text{and}\qquad 1 &= \int_d^1 \frac{x+t}{x+t} \dif\!t + \int_1^{1+d} \frac{x+t}{x+t} \dif\!t .
  \end{align*}
  Using the rewritten terms and simplifying leads to
  \begin{equation*}
    m_d(x) - m_d(x+1)
    = \int_d^1 \frac{\tfrac{1}{2}-t}{x+t} \dif\!t + \int_1^{1+d} \frac{\tfrac{1}{2} - (t-1)}{x+t} \dif\!t
  \end{equation*}
  and as \(d \in {[0,1]}\) holds,
  we thus arrive at the formula
  \begin{equation*}
    m_d(x) - m_d(x+1)
    = \int_d^{1+d} \frac{\tfrac{1}{2} - \groupb{t}}{x+t} \dif\!t . \qedhere
  \end{equation*}
\end{proof}

With the formula giving the differences of the logarithmic multiplicative mismatches \(m_d\)
as an integral, when \(d \in {[0,1]}\),
we can also give a formula for the logarithmic multiplicative mismatch \(m_d\)
in the form of an integral in this case.
\begin{lemma}\label{lemma:md-(0,1)-int}
  For any \(d \in {[0,1]}\) and \(x \in D_d\),
  we have
  \begin{equation*}
    m_d(x) = \int_{d}^{\infty} \frac{\frac{1}{2} - \groupb{t}}{x+t} \dif\!t ,
  \end{equation*}
  where the integral may either be considered as a Lebesgue integral
  or an improper Riemann integral.
\end{lemma}
\begin{proof}
  For any \(k \in \Nbbb\),
  we have
  \begin{equation*}
    m_d(x) - m_d(x+k) = \int_{d}^{k+d} \frac{\frac{1}{2} - \groupb{t}}{x+t} \dif\!t
  \end{equation*}
  by Lemma~\ref{lemma:md-diff}.
  Since we know from Lemma~\ref{lemma:g-Sd-asymp} that
  \begin{equation*}
    \lim_{k \to \infty} m_d(x+k) = 0
  \end{equation*}
  must hold, the assertion follows by letting \(k\) tend to infinity.
\end{proof}

Now,
we can derive the following formula for the logarithmic multiplicative mismatches \(m_d\)
using two integrals directly as a corollary of Lemmas~\ref{lemma:md-to-md} and~\ref{lemma:md-(0,1)-int}.
\begin{corollary}\label{corollary:md-int-half}
  For any \(d \in \Rbbb\) and \(x \in D_d\),
  we have
  \begin{equation*}
    m_d(x) = \int_{\frac{1}{2}}^{\infty} \frac{\frac{1}{2} - \groupb{t}}{x+t} \dif\!t + \int_{\frac{1}{2}}^{d} \frac{t - \tfrac{1}{2}}{x + t} \dif\!t ,
  \end{equation*}
  where the first integral may either be considered as a Lebesgue integral
  or an improper Riemann integral.
\end{corollary}
\begin{remark}
  Note that in order to be able to apply Lemmas~\ref{lemma:md-to-md} and~\ref{lemma:md-(0,1)-int}
  to prove Corollary~\ref{corollary:md-int-half},
  one needs to verify that \(D_d \cap D_{\frac{1}{2}} = D_d\) holds,
  which however is trivial.
\end{remark}

Moreover,
we can also rewrite Corollary~\ref{corollary:md-int-half} to give
the following class of formulas for the logarithmic multiplicative mismatches \(m_d\)
using two integrals.
\begin{lemma}\label{lemma:md-int}
  For any \(d \in \Rbbb\) and \(x \in D_d\),
  we have
  \begin{equation*}
    m_d(x) = \int_{d}^{\infty} \frac{\frac{1}{2} - \groupb{t}}{x+t} \dif\!t + \int_{c}^d \frac{\groupf{t}}{x + t} \dif\!t ,
  \end{equation*}
  where the first integral may either be considered as a Lebesgue integral
  or an improper Riemann integral
  and one may choose any \(c \in {[0,1]}\).
\end{lemma}
\begin{proof}
  We start with the formula
  \begin{equation*}
    m_d(x) = \int_{\frac{1}{2}}^{\infty} \frac{\frac{1}{2} - \groupb{t}}{x+t} \dif\!t + \int_{\frac{1}{2}}^{d} \frac{t - \tfrac{1}{2}}{x + t} \dif\!t
  \end{equation*}
  from Corollary~\ref{corollary:md-int-half}.
  Then,
  using
  \begin{equation*}
    0 = \int_{d}^{\frac{1}{2}} \frac{\frac{1}{2} - \groupb{t}}{x+t} \dif\!t + \int_{\frac{1}{2}}^{d} \frac{\frac{1}{2} - \groupb{t}}{x+t} \dif\!t ,
  \end{equation*}
  one arrives at
  \begin{equation*}
    m_d(x) = \int_{d}^{\infty} \frac{\frac{1}{2} - \groupb{t}}{x+t} \dif\!t + \int_{\frac{1}{2}}^{d} \frac{t - \groupb{t}}{x + t} \dif\!t ,
  \end{equation*}
  where one can also insert that \(t - \groupb{t} = \groupf{t}\).
  Remarking that
  \begin{equation*}
    0 = \int_{c}^{\frac{1}{2}} \frac{\groupf{t}}{x + t} \dif\!t
  \end{equation*}
  holds for any \(c \in {[0,1]}\) concludes the proof.
\end{proof}

Combining Corollary~\ref{corollary:md-int-half} and Lemma~\ref{lemma:md-int}
we thus arrive at the following integral identity version of Stirling's asymptotic formulas:
\begin{theorem}\label{theorem:md-int}
  For any \(d \in \Rbbb\) and \(x \in D_d\) as well as \(c \in {[0,1]}\),
  we have
  \begin{align*}
    \frac{\Pi(x)}{S_d(x)}
    &= \exp\group[\Bigg]{\int_{\frac{1}{2}}^{\infty} \frac{\frac{1}{2} - \groupb{t}}{x+t} \dif\!t + \int_{\frac{1}{2}}^{d} \frac{t - \tfrac{1}{2}}{x + t} \dif\!t} \\
    &= \exp\group[\Bigg]{\int_{d}^{\infty} \frac{\frac{1}{2} - \groupb{t}}{x+t} \dif\!t + \int_{c}^d \frac{\groupf{t}}{x + t} \dif\!t} ,
  \end{align*}
  which for \(d \in {[0,1]}\) and \(x \in D_d\) simplifies to
  \begin{equation*}
    \frac{\Pi(x)}{S_d(x)} = \exp\group[\Bigg]{\int_{d}^{\infty} \frac{\frac{1}{2} - \groupb{t}}{x+t} \dif\!t} .
  \end{equation*}
\end{theorem}

We note that the identities in Theorem~\ref{theorem:md-int} are generalisations of the identity
\begin{equation*}
  \Gamma(x+1)
  = \sqrt{2\pi} \group{x}^{x + \frac{1}{2}} e^{-x} \exp\group[\Bigg]{\int_{0}^{\infty} \frac{\frac{1}{2} - \groupb{t}}{x+t} \dif\!t}
\end{equation*}
which is found for example in \cite{Liu2007}.
Indeed, the identities generalise the already generalised form of the identity of \cite{Liu2007}
which can be found in \cite{Mortici2010} and deals with the case \(d \in {[0,1]}\) and \(x \in \Nbbb\).

\section{The Stirling-Gautschi Integral Identity}\label{sec:StirlingGautschiIdentity}

Finally, we can combine the previous results to give integral identities for the asymptotic formulas
of piecewise logarithmic interpolations of the pi function.
For any \(d \in \Rbbb\), we now also define the function
\(\widehat{m}_d \colon \widehat{D}_d \to \Rbbb\)
by setting
\begin{equation*}
  \widehat{m}_d(x, \alpha) \isdef \log \frac{\widehat{\Pi}(x, \alpha)}{S_d(x+\alpha)} .
\end{equation*}
Hence \(\widehat{m}_d(x, \alpha)\) is the logarithm of the multiplicative
mismatch between the piecewise logarithmic interpolations of the pi function
and the asymptotic formulas.

We now note that we may rewrite the logarithmic multiplicative mismatch \(\widehat{m}_d\)
using the logarithmic multiplicative mismatches \(m_d\) and \(\widehat{\imath}\)
from the previous Sections~\ref{sec:GautschiIdentity} and~\ref{sec:StirlingIdentity}.

\begin{lemma}\label{lemma:imd-int}
  For any \(d \in \Rbbb\) and \((x, \alpha) \in \widehat{D}_d\) as well as \(c \in {[0,1]}\),
  we have
  \begin{equation*}
    \widehat{m}_d(x, \alpha) = \widehat{\imath}(x, \alpha) + m_d(x + \alpha)
  \end{equation*}
  and therefore
  \begin{align*}
    \widehat{m}_d(x, \alpha)
    &= \int_1^{\infty} \frac{\phi_\alpha(t)}{x+t} \dif\!t + \int_{\frac{1}{2}}^{\infty} \frac{\frac{1}{2} - \groupb{t}}{x+\alpha + t} \dif\!t + \int_{\frac{1}{2}}^{d} \frac{t - \tfrac{1}{2}}{x+\alpha + t} \dif\!t \\
    &= \int_1^{\infty} \frac{\phi_\alpha(t)}{x+t} \dif\!t + \int_{d}^{\infty} \frac{\frac{1}{2} - \groupb{t}}{x+\alpha + t} \dif\!t + \int_{c}^d \frac{\groupf{t}}{x+\alpha + t} \dif\!t ,
  \end{align*}
  which for \(d \in {[0,1]}\) and \((x, \alpha) \in \widehat{D}_d\) simplifies to
  \begin{equation*}
    \widehat{m}_d(x, \alpha)
    = \int_1^{\infty} \frac{\phi_\alpha(t)}{x+t} \dif\!t + \int_{d}^{\infty} \frac{\frac{1}{2} - \groupb{t}}{x+\alpha + t} \dif\!t .
  \end{equation*}
\end{lemma}
\begin{proof}
  Since one can write
  \begin{equation*}
    \log \frac{\widehat{\Pi}(x, \alpha)}{S_d(x+\alpha)}
    = \log \frac{\widehat{\Pi}(x, \alpha)}{\Pi(x+\alpha)} + \log \frac{\Pi(x+\alpha)}{S_d(x+\alpha)}
  \end{equation*}
  it is clear that the first assertion holds, with which
  all other asserted formulas follow from the previous results.
\end{proof}

From Lemma~\ref{lemma:imd-int} 
we directly arrive at the following Stirling-Gautschi-type integral identities:
\begin{theorem}\label{theorem:imd-int}
  For any \(d \in \Rbbb\) and \((x, \alpha) \in \widehat{D}_d\) as well as \(c \in {[0,1]}\),
  we have
  \begin{align*}
    \frac{\widehat{\Pi}(x, \alpha)}{S_d(x+\alpha)}
    &= \exp\group[\Bigg]{\int_1^{\infty} \frac{\phi_\alpha(t)}{x+t} \dif\!t + \int_{\frac{1}{2}}^{\infty} \frac{\frac{1}{2} - \groupb{t}}{x\!+\!\alpha \!+\! t} \dif\!t + \int_{\frac{1}{2}}^{d} \frac{t - \tfrac{1}{2}}{x\!+\!\alpha \!+\! t} \dif\!t} \\
    &= \exp\group[\Bigg]{\int_1^{\infty} \frac{\phi_\alpha(t)}{x+t} \dif\!t + \int_{d}^{\infty} \frac{\frac{1}{2} - \groupb{t}}{x\!+\!\alpha \!+\! t} \dif\!t + \int_{c}^d \frac{\groupf{t}}{x\!+\!\alpha \!+\! t} \dif\!t} ,
  \end{align*}
  which for \(d \in {[0,1]}\) and \((x, \alpha) \in \widehat{D}_d\) simplifies to
  \begin{equation*}
    \frac{\widehat{\Pi}(x, \alpha)}{S_d(x+\alpha)}
    = \exp\group[\Bigg]{\int_1^{\infty} \frac{\phi_\alpha(t)}{x+t} \dif\!t + \int_{d}^{\infty} \frac{\frac{1}{2} - \groupb{t}}{x+\alpha + t} \dif\!t} .
  \end{equation*}
\end{theorem}

\section{A New Gautschi's Inequality Type Bound}

In this section we investigate two-sided bounds for the integral
appearing in Theorem~\ref{theorem:loginterpol},
the integral identity version of Gautschi's inequality.
Specifically,
we are interested in having a bound of the form
\begin{equation*}
  \underline{b_G}(x + \alpha)
  \leq \int_1^{\infty} \frac{\phi_\alpha(t)}{x+t} \dif\!t
  \leq \overline{b_G}(x + \alpha)
\end{equation*}
that holds for all \((x,\alpha) \in \widehat{D}\).

It is trivial to verify that for any \((x, \alpha) \in \widehat{D}\)
one has 
\begin{equation*}
  \int_1^{\infty} \frac{\phi_\alpha(t)}{x+t} \dif\!t \geq 0 .
\end{equation*}
Moreover,
when \(\alpha \in \groupb{0, 1}\),
\begin{equation*}
  \int_1^{\infty} \frac{\phi_\alpha(t)}{x+t} \dif\!t = 0 
\end{equation*}
holds for all \(x \in D\).
Therefore,
we can easily conclude that choosing
\begin{equation*}
  \underline{b_G}(y) = 0
\end{equation*}
yields the best possible lower bound so that
\begin{equation*}
  \underline{b_G}(x + \alpha)
  \leq \int_1^{\infty} \frac{\phi_\alpha(t)}{x+t} \dif\!t
\end{equation*}
holds for all \((x,\alpha) \in \widehat{D}\).

In order to facilitate deriving an upper bound,
we now supply the following lemma.
\begin{lemma}
  For any \(z \in \Rbbb_{\geq 1}\) and \(\alpha, t \in {[0, 1]}\),
  \begin{multline*}
    \frac{\alpha (1-\alpha)}{\groups[\big]{(z+\alpha) + \alpha (t-1)} \groups[\big]{(z+\alpha) + (1-\alpha) (1-t)}} \\
    \leq \frac{1}{4} \frac{1}{(z+\alpha) - \tfrac{5}{8}} \frac{1}{(z+\alpha) + \tfrac{3}{8}}
  \end{multline*}
  holds.
\end{lemma}
\begin{proof}
  It is clear that the assertion holds
  if and only if
  \begin{multline*}
    4 \alpha (1-\alpha) \groups[\big]{(z+\alpha) - \tfrac{5}{8}}\groups[\big]{(z+\alpha) + \tfrac{3}{8}} \\
    \leq \groups[\big]{(z+\alpha) + \alpha (t-1)} \groups[\big]{(z+\alpha) + (1-\alpha) (1-t)}
  \end{multline*}
  holds.
  However,
  multiplying this out and rearranging it
  yields that this in turn holds
  if and only if
  \begin{multline*}
    0 \leq \groups[\big]{1 - 4 \alpha (1-\alpha)} (z+\alpha)^2 + \groups[\big]{(1-2\alpha)(1-t) + \alpha (1-\alpha)} (z+\alpha) \\
    + \alpha (1-\alpha)\groups[\big]{\tfrac{15}{16} - (1-t)^2}
  \end{multline*}
  holds.
  If \(\alpha \leq \tfrac{1}{2}\),
  then we have \(1-2\alpha \geq 0\) and thus
  \begin{align*}
    \groups[\big]{1 - 4 \alpha (1-\alpha)} (z+\alpha)^2
    + \groups[\big]{\underbrace{(1-2\alpha)}_{\geq 0}\underbrace{(1-t)}_{\geq 0} + \alpha (1-\alpha)} \underbrace{(z+\alpha)}_{\geq 1+\alpha} \\
    \geq \groups[\Big]{\underbrace{\groups[\big]{1 - 4 \alpha (1-\alpha)}}_{\geq 0} \underbrace{(z+\alpha)}_{\geq 1+\alpha} + \alpha (1-\alpha)} \underbrace{(z+\alpha)}_{\geq 1+\alpha} \\
    \geq \alpha (1-\alpha) (1+\alpha) .
  \end{align*}
  Hence,
  this now implies the assertion for \(\alpha \leq \tfrac{1}{2}\),
  since
  \begin{align*}
    \groups[\big]{1 - 4 \alpha (1-\alpha)} (z+\alpha)^2 + \groups[\big]{(1-2\alpha)(1-t) + \alpha (1-\alpha)} (z+\alpha) &\\
    + \alpha (1-\alpha)\groups[\big]{\tfrac{15}{16} - (1-t)^2} &\\
    \geq \alpha (1-\alpha) (1+\alpha) + \alpha (1-\alpha)\groups[\big]{\tfrac{15}{16} - (1-t)^2} &\\
    = \alpha (1-\alpha) \groups[\big]{(1+\alpha) + \tfrac{15}{16} - (1-t)^2} &\\
    \geq \alpha (1-\alpha) \groups[\big]{\alpha + \tfrac{15}{16}} &\geq 0 .
  \end{align*}
  Otherwise,
  when \(\alpha > \tfrac{1}{2}\),
  we have \(-1 \leq 1-2\alpha < 0\) and thus
  \begin{align*}
    \groups[\big]{1 - 4 \alpha (1-\alpha)} (z+\alpha)^2
    + \groups[\big]{\underbrace{(1-2\alpha)}_{< 0}\underbrace{(1-t)}_{\leq 1} + \alpha (1-\alpha)} \underbrace{(z+\alpha)}_{\geq 1+\alpha} \\
    \geq \groups[\Big]{\underbrace{\groups[\big]{1 - 4 \alpha (1-\alpha)}}_{\geq 0} \underbrace{(z+\alpha)}_{\geq 1+\alpha} + (1-2\alpha) + \alpha (1-\alpha)} \underbrace{(z+\alpha)}_{\geq 1+\alpha} \\
    \geq \groups[\Big]{(1+\alpha) - 4 \alpha (1-\alpha)(1+\alpha) + 1 - 2\alpha + \alpha (1-\alpha)} \underbrace{(z+\alpha)}_{\geq 1+\alpha} \\
    = \groups{4 \alpha^3 - \alpha^2 - 4\alpha + 2} (z+\alpha) .
  \end{align*}
  It is straightforward to verify that \(4 \alpha^3 - \alpha^2 - 4\alpha + 2 \geq 0\) holds
  for \(\tfrac{1}{2} < \alpha \leq 1\).
  Therefore,
  we arrive at
  \begin{align*}
    \groups[\big]{1 - 4 \alpha (1-\alpha)} (z+\alpha)^2 + \groups[\big]{(1-2\alpha)(1-t) + \alpha (1-\alpha)} (z+\alpha) &\\
    + \alpha (1-\alpha)\groups[\big]{\tfrac{15}{16} - (1-t)^2} &\\
    \geq \groups{4 \alpha^3 - \alpha^2 - 4\alpha + 2} (1+\alpha) -\tfrac{1}{16} \alpha (1-\alpha) &\\
    = 4 \alpha^4 + 3 \alpha^3 - \tfrac{79}{16} \alpha^2 - \tfrac{33}{16} \alpha + 2 &\geq 0 ,
  \end{align*}
  when \(\tfrac{1}{2} < \alpha \leq 1\),
  which thus completes the proof of the assertion.
\end{proof}

With the preceding lemma it is easy to derive an upper bound,
given in the following lemma:
\begin{lemma}\label{lemma:loginterpol-bound}
  Define \(\overline{b_G} \colon \Rbbb_{\geq 0} \to \Rbbb\) by
  \begin{equation*}
    \overline{b_G}(y) \isdef \frac{1}{8y + 3} ,
  \end{equation*}
  then for all \((x,\alpha) \in \widehat{D}\) we have
  \begin{equation*}
    \underline{b_G}(x + \alpha)
    \leq \int_1^{\infty} \frac{\phi_\alpha(t)}{x+t} \dif\!t
    \leq \overline{b_G}(x + \alpha) .
  \end{equation*}
\end{lemma}
\begin{proof}
  We start by noting that
  \begin{align*}
    &\int_1^{\infty} \frac{\phi_\alpha(t)}{x+t} \dif\!t
    = \sum_{k=1}^{\infty} \int_0^{1} \frac{\phi_\alpha(t)}{x+k+t} \dif\!t \\
    &= \sum_{k=1}^{\infty} \int_0^{\alpha} \frac{1-\alpha}{x+k+t} \dif\!t + \int_{\alpha}^1 \frac{-\alpha}{x+k+t} \dif\!t \\
    &= \sum_{k=1}^{\infty} \int_0^1 \frac{\alpha(1-\alpha)}{x+k+\alpha+\alpha(t-1)} \dif\!t - \int_0^1 \frac{\alpha(1-\alpha)}{x+k+\alpha+(1-\alpha)(1-t)} \dif\!t \\
    &=  \sum_{k=1}^{\infty} \int_0^1 \frac{\alpha(1-\alpha)(1-t)}{\groups[\big]{x+k+\alpha+\alpha(t-1)}\groups[\big]{x+k+\alpha+(1-\alpha)(1-t)}} \dif\!t .
  \end{align*}
  Now,  since \(x+k \geq 1\), we can apply the previous lemma using \(z = x+k\)
  to arrive at
  \begin{multline*}
    \frac{\alpha(1-\alpha)}{\groups[\big]{x+k+\alpha+\alpha(t-1)}\groups[\big]{x+k+\alpha+(1-\alpha)(1-t)}} \\
    \leq \frac{1}{4} \frac{1}{(x+k+\alpha) - \tfrac{5}{8}} \frac{1}{(x+k+\alpha) + \tfrac{3}{8}}
  \end{multline*}
  and hence, as \(1-t \geq 0\), we have the assertion as
  \begin{align*}
    \int_1^{\infty} \frac{\phi_\alpha(t)}{x+t} \dif\!t
    &\leq \frac{1}{4} \int_0^1 (1-t) \dif\!t \sum_{k=1}^{\infty} \frac{1}{(x+\alpha)+k - \tfrac{5}{8}} \frac{1}{(x+\alpha)+k + \tfrac{3}{8}} \\
    &= \frac{1}{8} \sum_{k=1}^{\infty} \group[\bigg]{\frac{1}{(x+\alpha)+k - \tfrac{5}{8}} - \frac{1}{(x+\alpha)+k + \tfrac{3}{8}}} \\
    &= \frac{1}{8} \frac{1}{(x+\alpha) + \tfrac{3}{8}} . \qedhere
  \end{align*}
\end{proof}

An immediate consequence of Theorem~\ref{theorem:loginterpol}
and the above bounds in Lemma~\ref{lemma:loginterpol-bound}
is the following version of Gautschi's inequality:
\begin{theorem}\label{theorem:loginterpol-bound}
  For any \((x, \alpha) \in \widehat{D}\),
  we have
  \begin{equation*}
    1
    \leq \frac{\widehat{\Pi}(x, \alpha)}{\Pi(x + \alpha)}
    \leq \exp\group[\Bigg]{\frac{1}{8(x + \alpha) + 3}} ,
  \end{equation*}
  or restating the inequality in the form commonly used for stating Gautschi's inequality,
  \begin{equation*}
    y^{1-s}
    \leq \frac{\Gamma(y+1)}{\Gamma(y+s)}
    \leq y^{1-s} \exp\group[\Bigg]{\frac{1}{8(y + s) - 5}}
  \end{equation*}
  for any \(y \in \Rbbb_{\geq 1}\) and \(s \in {[0,1]}\).
\end{theorem}

\section{General Bounds for Stirling's Asymptotic Formulas}

We will now consider bounds for the terms
found in the integral identities of Stirling's approximation,
i.e. two of the integrals that are in Theorem~\ref{theorem:md-int}.
For this we first supply the following proposition:
\begin{proposition}
  For any \(y \in \Rbbb_{\geq 1}\) and \(t \in {[0, \frac{1}{2}]}\)
  one has
  \begin{equation*}
    (y - \tfrac{1}{2}) (y + \tfrac{1}{2})
    \leq y^2 - t^2
    \leq (y + \tfrac{\sqrt{5}-2}{2} - \tfrac{1}{2}) (y + \tfrac{\sqrt{5}-2}{2} + \tfrac{1}{2}) .
  \end{equation*}
\end{proposition}
\begin{proof}
  We start by noting that the left inequality is trivial,
  as
  \begin{equation*}
    (y - \tfrac{1}{2}) (y + \tfrac{1}{2})
    = y^2 - (\tfrac{1}{2})^2
    \leq y^2 - t^2 .
  \end{equation*}
  Moreover,
  for the right inequality we simply calculate
  \begin{align*}
    (y + \tfrac{\sqrt{5}-2}{2} - \tfrac{1}{2}) (y + \tfrac{\sqrt{5}-2}{2} + \tfrac{1}{2})
    &= y^2 + (\sqrt{5}-2) y + \tfrac{1}{4} (\sqrt{5}-2)^2 - \tfrac{1}{4} \\
    &\geq y^2 + (\sqrt{5}-2) + \tfrac{1}{4} (9 - 4\sqrt{5}) - \tfrac{1}{4} \\
    &= y^2 + \tfrac{9}{4} - 2 - \tfrac{1}{4} + \sqrt{5} - \sqrt{5} = y^2 . \qedhere
  \end{align*}
\end{proof}

With this proposition we can derive bounds of the form
\begin{equation*}
  \underline{b_{S}^{1/2}}(x)
  \leq \int_{\frac{1}{2}}^{\infty} \frac{\frac{1}{2} - \groupb{t}}{x + t} \dif\!t
  \leq \overline{b_{S}^{1/2}}(x)
\end{equation*}
as follows.
Note that this directly covers the case \(d = \frac{1}{2}\)
in Theorem~\ref{theorem:md-int}.

\begin{lemma}\label{lemma:md-halfbound}
  Define \(\underline{b_{S}^{1/2}}, \overline{b_{S}^{1/2}} \colon D_{1/2} \to \Rbbb\) by
  \begin{equation*}
    \underline{b_{S}^{1/2}}(x) \isdef -\frac{1}{24 x + 12}
    \qquad\text{and}\qquad
    \overline{b_{S}^{1/2}}(x) \isdef -\frac{1}{24 x + 12(\sqrt{5} - 1)} ,
  \end{equation*}
  then for all \(x \in D_{1/2}\) we have
  \begin{equation*}
    \underline{b_{S}^{1/2}}(x)
    \leq \int_{\frac{1}{2}}^{\infty} \frac{\frac{1}{2} - \groupb{t}}{x + t} \dif\!t
    \leq \overline{b_{S}^{1/2}}(x)
  \end{equation*}
  and all three terms of the inequalities are concave and strictly increasing in \(x\).
\end{lemma}
\begin{proof}
  We start by calculating
  \begin{align*}
    \int_{\frac{1}{2}}^{\infty} \frac{\frac{1}{2} - \groupb{t}}{x + t} \dif\!t
    &= \sum_{k=1}^\infty \int_{\frac{1}{2}}^{\frac{3}{2}} \frac{\frac{1}{2} - \groupb{t}}{x + k - 1 + t} \dif\!t \\
    &= \sum_{k=1}^\infty \group[\bigg]{\int_{\frac{1}{2}}^{1} \frac{\frac{1}{2} - t}{x + k - 1 + t} \dif\!t + \int_{1}^{\frac{3}{2}} \frac{\frac{1}{2} - (t-1)}{x + k - 1 + t} \dif\!t } \\
    &= \sum_{k=1}^\infty \group[\bigg]{\int_{0}^{\frac{1}{2}} \frac{t - \frac{1}{2}}{x + k - t} \dif\!t + \int_{0}^{\frac{1}{2}} \frac{\frac{1}{2} - t}{x + k + t} \dif\!t} \\
    &= \sum_{k=1}^\infty \int_{0}^{\frac{1}{2}} \frac{2 t^2 - t}{(x + k)^2 - t^2} \dif\!t .
  \end{align*}
  Since \(2 t^2 - t \leq 0\) holds,
  it is immediate from the last term in this calculation
  that the middle term of the inequalities is concave and strictly increasing in \(x\),
  while this is trivial for the two outer terms of the inequalities.
  Next, using \((x+k)^2 - t^2 \leq (x+k + \tfrac{\sqrt{5}-2}{2} - \tfrac{1}{2}) (x+k + \tfrac{\sqrt{5}-2}{2} + \tfrac{1}{2})\)
  as well as \(2 t^2 - t \leq 0\),
  we have
  \begin{align*}
    \int_{\frac{1}{2}}^{\infty} \frac{\frac{1}{2} - \groupb{t}}{x + t} \dif\!t
    &\leq \group[\bigg]{\int_{0}^{\frac{1}{2}} 2 t^2 - t \dif\!t} \sum_{k=1}^\infty \frac{1}{(x+k)^2} \\
    &\leq -\frac{1}{24} \sum_{k=1}^\infty \frac{1}{(x+k+ \tfrac{\sqrt{5}-2}{2} - \tfrac{1}{2})(x+k + \tfrac{\sqrt{5}-2}{2} + \tfrac{1}{2})} \\
    &= -\frac{1}{24} \sum_{k=1}^\infty \group[\bigg]{\frac{1}{x+k + \tfrac{\sqrt{5}-2}{2} - \tfrac{1}{2}} - \frac{1}{x+k + \tfrac{\sqrt{5}-2}{2} + \tfrac{1}{2}}} \\
    &= -\frac{1}{24} \frac{1}{x+\tfrac{\sqrt{5}-1}{2}}
    = \frac{1}{24x + 12(\sqrt{5} - 1)}
  \end{align*}
  and, using \((x+k)^2 - t^2 \geq (x+k - \tfrac{1}{2})(x+k + \tfrac{1}{2}) \) and \(2 t^2 - t \leq 0\), also
  \begin{align*}
    \int_{\frac{1}{2}}^{\infty} \frac{\frac{1}{2} - \groupb{t}}{x + t} \dif\!t
    &\geq \group[\bigg]{\int_{0}^{\frac{1}{2}} 2 t^2 - t \dif\!t} \sum_{k=1}^\infty \frac{1}{(x+k)^2 - \tfrac{1}{4}} \\
    &= -\frac{1}{24} \sum_{k=1}^\infty \frac{1}{(x+k - \frac{1}{2}) (x+k + \frac{1}{2})} \\
    &= -\frac{1}{24} \sum_{k=1}^\infty \group[\bigg]{\frac{1}{x+k - \frac{1}{2}} - \frac{1}{x+k + \frac{1}{2}}} \\
    &= -\frac{1}{24} \frac{1}{x + \frac{1}{2}}
    = -\frac{1}{24x + 12} . \qedhere
  \end{align*}
\end{proof}

If \(d \neq \frac{1}{2}\),
then to be able to use Lemma~\ref{lemma:md-halfbound}
in Theorem~\ref{theorem:md-int} we also need 
to consider bounds for the other integral,
that is we need to give bounds of the form
\begin{equation*}
  \underline{q_{S}^{d}}(x)
  \leq \int_{\frac{1}{2}}^{d} \frac{t - \tfrac{1}{2}}{x + t} \dif\!t
  \leq \overline{q_{S}^{d}}(x) .
\end{equation*}

\begin{lemma}\label{lemma:md-shiftbound}
  For any \(d \in \Rbbb\)
  define \(\underline{q_{S}^{d}}, \overline{q_{S}^{d}} \colon D_{d} \to \Rbbb\) by
  \begin{equation*}
    \underline{q_{S}^{d}}(x) \isdef \frac{(d-\frac{1}{2})^2}{2x + \max\groupb{2d,1}}
    \qquad\text{and}\qquad
    \overline{q_{S}^{d}}(x) \isdef \frac{(d-\frac{1}{2})^2}{2x + \min\groupb{2d,1}} ,
  \end{equation*}
  then for all \(x \in D_d\) we have
  \begin{equation*}
    \underline{q_{S}^{d}}(x)
    \leq \int_{\frac{1}{2}}^{d} \frac{t - \tfrac{1}{2}}{x + t} \dif\!t
    \leq \overline{q_{S}^{d}}(x) .
  \end{equation*}
\end{lemma}
\begin{proof}
  We first assume that \(d \geq \tfrac{1}{2}\).
  Then we may calculate
  \begin{equation*}
    \int_{\frac{1}{2}}^{d} \frac{t - \tfrac{1}{2}}{x + t} \dif\!t
    = \int_{0}^{d-\frac{1}{2}} \frac{t}{x+\tfrac{1}{2} + t} \dif\!t ,
  \end{equation*}
  from which the asserted bounds follow from
  \begin{equation*}
    \int_{0}^{d-\frac{1}{2}} \frac{t}{x+\tfrac{1}{2} + t} \dif\!t
    \leq \frac{1}{x+\tfrac{1}{2}} \int_{0}^{d-\frac{1}{2}} t \dif\!t = \frac{(d-\frac{1}{2})^2}{2x + 1}
  \end{equation*}
  and
  \begin{equation*}
    \int_{0}^{d-\frac{1}{2}} \frac{t}{x+\tfrac{1}{2} + t} \dif\!t
    \geq \frac{1}{x+d} \int_{0}^{d-\frac{1}{2}} t \dif\!t = \frac{(d-\frac{1}{2})^2}{2x + 2d} .
  \end{equation*}
  Analogously,
  we otherwise assume that \(d \leq \tfrac{1}{2}\).
  In this case we may calculate
  \begin{equation*}
    \int_{\frac{1}{2}}^{d} \frac{t - \tfrac{1}{2}}{x + t} \dif\!t
    = \int_{d-\frac{1}{2}}^{0} \frac{-t}{x+\tfrac{1}{2} + t} \dif\!t ,
  \end{equation*}
  from which the asserted bounds now follow owing to
  \begin{equation*}
    \int_{d-\frac{1}{2}}^{0} \frac{-t}{x+\tfrac{1}{2} + t} \dif\!t
    \leq \frac{1}{x+d} \int_{d-\frac{1}{2}}^{0} -t \dif\!t = \frac{(d-\frac{1}{2})^2}{2x + 2d}
  \end{equation*}
  and
  \begin{equation*}
    \int_{d-\frac{1}{2}}^{0} \frac{-t}{x+\tfrac{1}{2} + t} \dif\!t
    \geq \frac{1}{x+\tfrac{1}{2}} \int_{d-\frac{1}{2}}^{0} -t \dif\!t = \frac{(d-\frac{1}{2})^2}{2x + 1} . \qedhere
  \end{equation*}
\end{proof}

Combining Theorem~\ref{theorem:md-int} with
the Lemmas~\ref{lemma:md-halfbound} and~\ref{lemma:md-shiftbound}
now yields the following two-sided bounds version of Stirling's asymptotic formulas:
\begin{theorem}\label{theorem:md-bound}
  For any \(d \in \Rbbb\) and \(x \in D_d\),
  we have
  \begin{align*}
    &\exp\group[\bigg]{-\frac{1}{24x+12} + \frac{(d-\frac{1}{2})^2}{2x + \max\groupb{2d,1}}} \\
    &\qquad\quad\leq \frac{\Pi(x)}{S_d(x)} \\
    &\qquad\qquad\qquad\leq \exp\group[\bigg]{-\frac{1}{24x+12(\sqrt{5} - 1)} + \frac{(d-\frac{1}{2})^2}{2x + \min\groupb{2d,1}}} .
  \end{align*}
\end{theorem}

We note that Theorem~\ref{theorem:md-bound} amounts to a generalisation
and improvement of the known two-sided bounds,
\begin{equation*}
  \exp\group[\bigg]{-\frac{1}{24x}}
  \leq \frac{\Pi(x)}{S_{1/2}(x)}
  \leq \exp\group[\bigg]{-\frac{1}{24x+24+3x^{-1}}}
\end{equation*}
found in \cite{BBE11},
from \(d = \tfrac{1}{2}\) to a general \(d \in \Rbbb\).
Specifically,
Theorem~\ref{theorem:md-bound} also relates to \cite{Mortici},
where it is shown that \(d = \tfrac{1}{2} \pm \tfrac{\sqrt{3}}{6}\)
asymptotically are the best choices.
Indeed this is obvious from Theorem~\ref{theorem:md-bound},
as with these two choices,
one has that
\begin{equation*}
  \frac{2x + \max\groupb{2d,1}}{(d-\frac{1}{2})^2} \sim 24x+\mathrm{const.} \sim \frac{2x + \min\groupb{2d,1}}{(d-\frac{1}{2})^2}
\end{equation*}
holds for \(x \to \infty\).

\section{Improved Bounds for Two of Stirling's Asymptotic Formulas}

As will be evident from the numerical comparisons in Section~\ref{sec:numcomp},
the bounds in Theorem~\ref{theorem:md-bound} are not as tight
when \(d \neq \tfrac{1}{2}\) as when \(d = \tfrac{1}{2}\).
For example, bounds already established by Robbins in \cite{Robbins}
for \(d = 0\) and positive integers outperform those in Theorem~\ref{theorem:md-bound}.

Therefore, we will now also show how one may use the one integral version
of Stirling's asymptotic formula found in Theorem~\ref{theorem:md-int}
to achieve better bounds when \(d = 0\) or \(d = 1\).
We first supply the following proposition,
to simplify considering these two cases:

\begin{proposition}
  For any \(y \in \Rbbb_{\geq 1/2}\) and \(t \in {[0, \frac{1}{2}]}\)
  one has
  \begin{equation*}
    (y - \tfrac{1}{2}) (y + \tfrac{1}{2})
    \leq y^2 - t^2
    \leq (y + \tfrac{\sqrt{2}-1}{2} - \tfrac{1}{2}) (y + \tfrac{\sqrt{2}-1}{2} + \tfrac{1}{2}) ,
  \end{equation*}
  where the right inequality can be strengthened to
  \begin{equation*}
    y^2 - t^2
    \leq (y + \tfrac{\sqrt{10}-3}{2} - \tfrac{1}{2}) (y + \tfrac{\sqrt{10}-3}{2} + \tfrac{1}{2}) ,
  \end{equation*}
  when \(y \in \Rbbb_{\geq 3/2}\).
\end{proposition}
\begin{proof}
  We start by noting that the left inequality is trivial,
  as
  \begin{equation*}
    (y - \tfrac{1}{2}) (y + \tfrac{1}{2})
    = y^2 - (\tfrac{1}{2})^2
    \leq y^2 - t^2 .
  \end{equation*}
  Moreover,
  for the right inequality we simply calculate
  \begin{align*}
    &(y + \tfrac{\sqrt{2}-1}{2} - \tfrac{1}{2}) (y + \tfrac{\sqrt{2}-1}{2} + \tfrac{1}{2}) \\
    &\qquad = y^2 + (\sqrt{2}-1) y + \tfrac{1}{4} (\sqrt{2}-1)^2 - \tfrac{1}{4} \\
    &\qquad \geq y^2 + \tfrac{1}{2}(\sqrt{2}-1) + \tfrac{1}{4} (3 - 2\sqrt{2}) - \tfrac{1}{4} \\
    &\qquad = y^2 + \tfrac{3}{4} - \tfrac{1}{2} - \tfrac{1}{4} + \tfrac{\sqrt{2}}{2} - \tfrac{\sqrt{2}}{2} = y^2 .
  \end{align*}
  Lastly, for the strengthened right inequality we directly calculate
  \begin{align*}
    &(y + \tfrac{\sqrt{10}-3}{2} - \tfrac{1}{2}) (y + \tfrac{\sqrt{10}-3}{2} + \tfrac{1}{2}) \\
    &\qquad = y^2 + (\sqrt{10}-3) y + \tfrac{1}{4} (\sqrt{10}-3)^2 - \tfrac{1}{4} \\
    &\qquad \geq y^2 + \tfrac{3}{2}(\sqrt{10}-3) + \tfrac{1}{4} (19 - 6\sqrt{10}) - \tfrac{1}{4} \\
    &\qquad = y^2 + \tfrac{19}{4} - \tfrac{9}{2} - \tfrac{1}{4} + \tfrac{3\sqrt{10}}{2} - \tfrac{3\sqrt{10}}{2} = y^2 . \qedhere
  \end{align*}
\end{proof}

We now first consider the case of \(d = 0\),
for which we have the following result:

\begin{lemma}\label{lemma:md-zerobound}
  Define \(\underline{b_{S}^{0}}, \overline{b_{S}^{0}} \colon D_{0} \to \Rbbb\) by
  \begin{equation*}
    \underline{b_{S}^{0}}(x) \isdef \frac{1}{12 x + 6 (\sqrt{2}-1)}
    \qquad\text{and}\qquad
    \overline{b_{S}^{0}}(x) \isdef \frac{1}{12 x} ,
  \end{equation*}
  then for all \(x \in D_{0}\) we have
  \begin{equation*}
    \underline{b_{S}^{0}}(x)
    \leq \int_{0}^{\infty} \frac{\frac{1}{2} - \groupb{t}}{x + t} \dif\!t
    \leq \overline{b_{S}^{0}}(x)
  \end{equation*}
   and all three terms of the inequalities are convex and strictly decreasing in \(x\).
\end{lemma}
\begin{proof}
  We start by calculating
  \begin{align*}
    \int_{0}^{\infty} \frac{\frac{1}{2} - \groupb{t}}{x + t} \dif\!t
    &= \sum_{k=0}^\infty \int_{0}^{1} \frac{\frac{1}{2} - \groupb{t}}{x + k + t} \dif\!t \\
    &= \sum_{k=0}^\infty \group[\bigg]{\int_{0}^{\frac{1}{2}} \frac{\frac{1}{2} - t}{x + k + t} \dif\!t + \int_{\frac{1}{2}}^{1} \frac{\frac{1}{2} - t}{x + k + t} \dif\!t } \\
    &= \sum_{k=0}^\infty \group[\bigg]{\int_{0}^{\frac{1}{2}} \frac{t}{x + k + \frac{1}{2} - t} \dif\!t + \int_{0}^{\frac{1}{2}} \frac{- t}{x + k + \frac{1}{2} + t} \dif\!t} \\
    &= \sum_{k=0}^\infty \int_{0}^{\frac{1}{2}} \frac{2 t^2}{(x + k + \frac{1}{2})^2 - t^2} \dif\!t .
  \end{align*}
  Since \(2 t^2\geq 0\) holds,
  it is immediate from the last term in this calculation
  that the middle term of the inequalities is convex and strictly decreasing in \(x\),
  while this is trivial for the two outer terms of the inequalities.
  Now, using \(2 t^2\geq 0\) and \((x+k+ \frac{1}{2})^2 - t^2 \geq (x+k) (x+k + 1)\),
  we have
  \begin{align*}
    \int_{0}^{\infty} \frac{\frac{1}{2} - \groupb{t}}{x + t} \dif\!t
    &\leq \group[\bigg]{\int_{0}^{\frac{1}{2}} 2 t^2 \dif\!t} \sum_{k=0}^\infty \frac{1}{(x+k)(x+k+1)} \\
    &= \frac{1}{12} \sum_{k=0}^\infty \group[\bigg]{\frac{1}{x+k} - \frac{1}{x+k + 1}}
    = \frac{1}{12} \frac{1}{x} = \frac{1}{12 x}
  \end{align*}
  and, using \((x+k+ \frac{1}{2})^2 - t^2 \leq (x+k + \tfrac{\sqrt{2}-1}{2}) (x+k + \tfrac{\sqrt{2}-1}{2} + 1)\)
  as well as \(2 t^2 \geq 0\), also
  \begin{align*}
    &\int_{0}^{\infty} \frac{\frac{1}{2} - \groupb{t}}{x + t} \dif\!t \\
    &\qquad \geq \group[\bigg]{\int_{0}^{\frac{1}{2}} 2 t^2\dif\!t} \sum_{k=0}^\infty \frac{1}{(x+k + \tfrac{\sqrt{2}-1}{2}) (x+k + \tfrac{\sqrt{2}-1}{2} + 1)} \\
    &\qquad = \frac{1}{12} \sum_{k=0}^\infty \group[\bigg]{\frac{1}{x+k + \tfrac{\sqrt{2}-1}{2}} - \frac{1}{x+k + \tfrac{\sqrt{2}-1}{2} + 1}} \\
    &\qquad = \frac{1}{12} \frac{1}{x + \tfrac{\sqrt{2}-1}{2}}
    = \frac{1}{12 x + 6 (\sqrt{2}-1)} . \qedhere
  \end{align*}
\end{proof}

Similarily, for the case of \(d = 1\),
we have the following result:

\begin{lemma}\label{lemma:md-onebound}
  Define \(\underline{b_{S}^{1}}, \overline{b_{S}^{1}} \colon D_{0} \to \Rbbb\) by
  \begin{equation*}
    \underline{b_{S}^{1}}(x) \isdef \frac{1}{12 x + 6 (\sqrt{10}-1)}
    \qquad\text{and}\qquad
    \overline{b_{S}^{1}}(x) \isdef \frac{1}{12 x + 12} ,
  \end{equation*}
  then for all \(x \in D_{0}\) we have
  \begin{equation*}
    \underline{b_{S}^{1}}(x)
    \leq \int_{1}^{\infty} \frac{\frac{1}{2} - \groupb{t}}{x + t} \dif\!t
    \leq \overline{b_{S}^{1}}(x)
  \end{equation*}
  and all three terms of the inequalities are convex and strictly decreasing in \(x\).
\end{lemma}
\begin{proof}
  We start by calculating
  \begin{align*}
    \int_{1}^{\infty} \frac{\frac{1}{2} - \groupb{t}}{x + t} \dif\!t
    &= \sum_{k=1}^\infty \int_{0}^{1} \frac{\frac{1}{2} - \groupb{t}}{x + k + t} \dif\!t \\
    &= \sum_{k=1}^\infty \int_{0}^{\frac{1}{2}} \frac{2 t^2}{(x + k + \frac{1}{2})^2 - t^2} \dif\!t ,
  \end{align*}
  analogously to the start in the proof of Lemma~\ref{lemma:md-zerobound}.
  Again, since \(2 t^2\geq 0\) holds,
  it is immediate from the last term in this calculation
  that the middle term of the inequalities is convex and strictly decreasing in \(x\),
  while this is trivial for the two outer terms of the inequalities.
  Next, using \(2 t^2\geq 0\) as well as \((x+k+ \frac{1}{2})^2 - t^2 \geq (x+k) (x+k + 1)\),
  we have
  \begin{align*}
    \int_{1}^{\infty} \frac{\frac{1}{2} - \groupb{t}}{x + t} \dif\!t
    &\leq \group[\bigg]{\int_{0}^{\frac{1}{2}} 2 t^2 \dif\!t} \sum_{k=1}^\infty \frac{1}{(x+k)(x+k+1)} \\
    &= \frac{1}{12} \sum_{k=1}^\infty \group[\bigg]{\frac{1}{x+k} - \frac{1}{x+k + 1}} \\
    &= \frac{1}{12} \frac{1}{x+1} = \frac{1}{12 x + 12}
  \end{align*}
  and, using \((x+k+ \frac{1}{2})^2 - t^2 \leq (x+k + \tfrac{\sqrt{10}-3}{2}) (x+k + \tfrac{\sqrt{10}-3}{2} + 1)\)
  as well as \(2 t^2 \geq 0\), also
  \begin{align*}
    &\int_{1}^{\infty} \frac{\frac{1}{2} - \groupb{t}}{x + t} \dif\!t \\
    &\qquad \geq \group[\bigg]{\int_{0}^{\frac{1}{2}} 2 t^2\dif\!t} \sum_{k=1}^\infty \frac{1}{(x+k + \tfrac{\sqrt{10}-3}{2}) (x+k + \tfrac{\sqrt{10}-3}{2} + 1)} \\
    &\qquad = \frac{1}{12} \sum_{k=1}^\infty \group[\bigg]{\frac{1}{x+k + \tfrac{\sqrt{10}-3}{2}} - \frac{1}{x+k + \tfrac{\sqrt{10}-3}{2} + 1}} \\
    &\qquad = \frac{1}{12} \frac{1}{x + \tfrac{\sqrt{10}-3}{2} + 1}
    = \frac{1}{12 x + 6 (\sqrt{10}-1)} . \qedhere
  \end{align*}
\end{proof}

Theorem~\ref{theorem:md-int}, together with
the Lemmas~\ref{lemma:md-zerobound} and~\ref{lemma:md-onebound},
now yields the following improved two-sided bounds for Stirling's asymptotic formulas
with \(d = 0\) and \(d = 1\):
\begin{theorem}\label{theorem:md-zeroonebound}
  For any \(x \in D_0\), we have
  \begin{equation*}
    \exp\group[\bigg]{\frac{1}{12 x + 6(\sqrt{2}-1)}}
    \leq \frac{\Pi(x)}{S_0(x)}
    \leq \exp\group[\bigg]{\frac{1}{12 x}} ,
  \end{equation*}
  while for any \(x \in D_1\)
  \begin{equation*}
    \exp\group[\bigg]{\frac{1}{12 x + 6(\sqrt{10}-1)}}
    \leq \frac{\Pi(x)}{S_1(x)}
    \leq \exp\group[\bigg]{\frac{1}{12 x + 12}}
  \end{equation*}
  holds.
\end{theorem}

\begin{remark}
  At first glance one sees that the lower bound
  \begin{equation*}
    \exp\group[\bigg]{\frac{1}{12 x + 6(\sqrt{2}-1)}}
    \leq \frac{\Pi(x)}{S_0(x)}
    \leq \exp\group[\bigg]{\frac{1}{12 x}}
  \end{equation*}
  for any \(x \in D_0\) in Theorem~\ref{theorem:md-zeroonebound}
  is weaker than the bound
  \begin{equation*}
    \exp\group[\bigg]{\frac{1}{12 n + 1}}
    \leq \frac{\Pi(n)}{S_0(n)}
    \leq \exp\group[\bigg]{\frac{1}{12 n}}
  \end{equation*}
  already established by Robbins in \cite{Robbins}.
  However, it is important to note that this bound is
  proven in \cite{Robbins} only for positive integers \(n\).
  Indeed, if one requires \(x \geq 1\) rather than
  \(x > 0\) to hold in Lemma~\ref{lemma:md-zerobound},
  then one can replace the term \(\tfrac{\sqrt{2}-1}{2}\)
  by \(\tfrac{\sqrt{10}-3}{2}\)
  and hence (minusculely) improves the version of Robbins bound:
  \begin{equation*}
    \exp\group[\bigg]{\frac{1}{12 x + 1}}
    < \exp\group[\bigg]{\frac{1}{12 x + 6(\sqrt{10}-3)}}
    \leq \frac{\Pi(x)}{S_0(x)}
  \end{equation*}
  holds for all \(x \geq 1\),
  as \(1 > 6(\sqrt{10}-3)\).
\end{remark}

\section{Improved Stirling-Gautschi Bounds of Burnside's Type}

It is obvious that Theorem~\ref{theorem:imd-int},
together with the upper bound found in Lemma~\ref{lemma:loginterpol-bound}
and its corresponding trivial lower bound,
as well as the two-sided bound of Lemmas~\ref{lemma:md-halfbound}
and~\ref{lemma:md-shiftbound}
or either~\ref{lemma:md-zerobound} or~\ref{lemma:md-onebound}
will turn the Stirling-Gautschi inequality
into two-sided Stirling-Gautschi bounds.
For example, restricting ourselves to \(d = \frac{1}{2}\)
we immediately arrive at the two-sided Stirling-Gautschi bounds of Burnside's type,
\begin{align*}
  &\exp\group[\bigg]{-\frac{1}{24(x + \alpha) + 12}} \\
  &\qquad \leq \frac{\widehat{\Pi}(x, \alpha)}{S_{1/2}(x+\alpha)} \\
  &\qquad\qquad\quad \leq \exp\group[\bigg]{\frac{1}{8(x + \alpha) + 3} - \frac{1}{24(x + \alpha) + 12(\sqrt{5}-1)}}
\end{align*}
for all \((x, \alpha) \in \widehat{D}_{1/2}\).
The expression for the upper bound can be simplified,
while keeping its asymptotic behaviour but worsening
the bound near zero,
using the following lemma:
\begin{lemma}\label{lemma:badburnsidestirlinggautschi}
  For all \(y \in \Rbbb_{\geq 0}\) one has
  \begin{equation*}
    \frac{1}{8y + 3} - \frac{1}{24y + 12(\sqrt{5}-1)}
    \leq \overline{b_{SG}}(y)
    \isdef \frac{1}{12y + 3} .
  \end{equation*}
\end{lemma}
\begin{proof}
  The inequality
  \begin{equation*}
    \frac{1}{8y + 3} - \frac{1}{24y + 12(\sqrt{5}-1)}
    \leq \frac{1}{12y + 3}
  \end{equation*}
  is equivalent to the inequality
  \begin{equation*}
    (16y + 12\sqrt{5} - 15)(12y + 3)
    \leq (8y + 3)(24y + 12\sqrt{5} - 12) .
  \end{equation*}
  However, simply multiplying this out gives
  \begin{multline*}
    192 y^2 + (144\sqrt{5}-132) y + 36\sqrt{5}-45 \\
    \leq 192 y^2 + (96\sqrt{5}-24) y + 36\sqrt{5}-36 ,
  \end{multline*}
  which is equivalent to the inequality
  \begin{equation*}
    0 \leq (108 - 48\sqrt{5}) y + 9
  \end{equation*}
  which is obviously true for all \(y \in \Rbbb_{\geq 0}\),
  as \(108 > 48\sqrt{5} \approx 107.331\).
\end{proof}
Using the preceeding lemma,
we thus have the simplified two-sided Stirling-Gautschi bounds of Burnside's type,
that is
\begin{equation*}
  \exp\group[\bigg]{-\frac{1}{24(x + \alpha) + 12}}
  \leq \frac{\widehat{\Pi}(x, \alpha)}{S_{1/2}(x+\alpha)}
  \leq \exp\group[\bigg]{\frac{1}{12(x + \alpha) + 3}}
\end{equation*}
holds for all \((x, \alpha) \in \widehat{D}_{1/2}\).
However,
we will see in the numerical comparisons
that the upper bound above is quite good asymptotically
but this is not the case near zero.

To address this,
we therefore derive a stricter upper bound.
For this, we first note that Burnside's asymptotic Stirling formula
has the following special behaviour when considering it for
piecewise logarithmic interpolations of the pi function:
\begin{lemma}\label{lemma:burnsidestirlinggautschi}
  For any \(x \in D_{1/2}\),
  the map
  \begin{equation*}
    {[0,1]} \to \Rbbb, \alpha \mapsto \widehat{m}_{1/2}(x, \alpha)
  \end{equation*}
  is concave and we have that
  \begin{equation*}
    \argmin_{\alpha \in {[0,1]}} \widehat{m}_{1/2}(x, \alpha)
    = 0
    \qquad\text{and}\qquad
    \argmax_{\alpha \in {[0,1]}} \widehat{m}_{1/2}(x, \alpha)
    = \tfrac{1}{2}
  \end{equation*}
  hold.
\end{lemma}
\begin{proof}
  We start by calculating the partial derivative of \(\widehat{m}_{1/2}\)
  with respect to \(\alpha\),
  which is
  \begin{equation*}
    \frac{\partial}{\partial \alpha} \widehat{m}_{1/2}(x, \alpha)
    = \log(x+1) - \log(x+\alpha+\tfrac{1}{2}) .
  \end{equation*}
  From this we immediately see
  that the function \(\alpha \mapsto \widehat{m}_{1/2}(x, \alpha)\)
  is increasing on \([0, \tfrac{1}{2})\), decreasing on \((\tfrac{1}{2}, 1]\)
  and has a critical point at \(\alpha = \tfrac{1}{2}\).
  Additionally,
  it is clear that the second partial derivative of \(\widehat{m}_{1/2}\)
  with respect to \(\alpha\)\ is
  \begin{equation*}
    \frac{\partial^2}{\partial \alpha^2} \widehat{m}_{1/2}(x, \alpha)
    = - \frac{1}{x+\alpha+\tfrac{1}{2}} .
  \end{equation*}
  This thus proves that the map \(\alpha \mapsto \widehat{m}_{1/2}(x, \alpha)\)
  is concave as well as that 
  \begin{equation*}
    \argmax_{\alpha \in {[0,1]}} \widehat{m}_{1/2}(x, \alpha) = \tfrac{1}{2}
    \qquad\text{and}\qquad
    \argmin_{\alpha \in {[0,1]}} \widehat{m}_{1/2}(x, \alpha) \in \{0,1\}
  \end{equation*}
  hold.
  Finally, from
  \begin{equation*}
    \widehat{m}_{1/2}(x, \alpha)
    = \widehat{m}_{1/2}(x, \tfrac{1}{2}) + \int_{\frac{1}{2}}^{\alpha} \log(x+1) - \log(x+t+\tfrac{1}{2}) \dif\!t ,
  \end{equation*}
  one obviously derives that
  \begin{equation*}
    \widehat{m}_{1/2}(x, \alpha)
    < \widehat{m}_{1/2}(x, 1-\alpha)
  \end{equation*}
  holds for all \(\alpha \in [0,\tfrac{1}{2})\)
  and arrives at the asserted
  \begin{equation*}
    \argmin_{\alpha \in {[0,1]}} \widehat{m}_{1/2}(x, \alpha)
    = 0 . \qedhere
  \end{equation*}
\end{proof}

Lemma~\ref{lemma:burnsidestirlinggautschi} implies
that investigating \(\widehat{m}_{1/2}(x, \tfrac{1}{2})\)
could possibly help to get a tighter upper bound.
Indeed,
one arrives at the following result:
\begin{lemma}\label{lemma:burnsidestirlinggautschihalfbound}
  For any \(x \in D_{1/2}\) one has
  \begin{equation*}
    \widehat{m}_{1/2}(x, \tfrac{1}{2})
    \leq \frac{1}{12 x + 12} ,
  \end{equation*}
  and both sides of the inequality are convex and strictly decreasing in \(x\).
\end{lemma}
\begin{proof}
  We start by noting that from Lemma~\ref{lemma:imd-int} we have
  \begin{equation*}
    \widehat{m}_{1/2}(x, \tfrac{1}{2})
    = \int_1^{\infty} \frac{\phi_{1/2}(t)}{x+t} \dif\!t + \int_{\tfrac{1}{2}}^{\infty} \frac{\frac{1}{2} - \groupb{t}}{x+\tfrac{1}{2} + t} \dif\!t . 
  \end{equation*}
  It is easy to see that the two integrals on the right-hand side simplify to give
  \begin{equation*}
    \widehat{m}_{1/2}(x, \tfrac{1}{2})
    = \int_1^{\infty} \frac{\frac{1}{2} - \groupb{t}}{x + t} \dif\!t . 
  \end{equation*}
  Now, simply applying Lemma~\ref{lemma:md-onebound} gives
  \begin{equation*}
    \widehat{m}_{1/2}(x, \tfrac{1}{2})
    \leq \frac{1}{12 x + 12},
  \end{equation*}
  and also the fact that both sides of the inequality are convex and strictly decreasing in \(x\).
\end{proof}

Lemma~\ref{lemma:burnsidestirlinggautschi}
and Lemma~\ref{lemma:burnsidestirlinggautschihalfbound}
directly give that
\begin{equation*}
  \widehat{m}_{1/2}(x, \alpha)
  \leq \widehat{m}_{1/2}(x, \tfrac{1}{2})
  \leq \frac{1}{12 x + 12}
  \leq \frac{1}{12 (x+\alpha)}
\end{equation*}
holds for every \((x, \alpha) \in \widehat{D}_{1/2}\).
It is clear that the last bound is unsatisfactory
as it is worse than the bound from Lemma~\ref{lemma:badburnsidestirlinggautschi},
but we can improve it as follows:
\begin{lemma}\label{lemma:goodburnsidestirlinggautschi}
  For every \((x, \alpha) \in \widehat{D}_{1/2}\) we have that
  \begin{equation*}
    \widehat{m}_{1/2}(x, \alpha)
    \leq \frac{1}{12 (x+\alpha) + K_x}
  \end{equation*}
  holds,
  where
  \begin{equation*}
    K_x = 12 - 6\group[\bigg]{(x+1) \exp\group[\bigg]{\frac{1}{12 (x+1)^2}}  - x}
  \end{equation*}
  is increasing in \(x\).
\end{lemma}
\begin{proof}
  We know that
  \begin{equation*}
    \widehat{m}_{1/2}(x, \alpha)
    \leq \widehat{m}_{1/2}(x, \tfrac{1}{2})
    \leq \frac{1}{12 x + 12}
  \end{equation*}
  holds.
  Note that we also know that
  \begin{equation*}
    \alpha \mapsto \widehat{m}_{1/2}(x, \alpha)
  \end{equation*}
  is a concave function for \(\alpha \in [0,1]\)
  and that
  \begin{equation*}
    x \mapsto \frac{1}{12 x + 12}
  \end{equation*}
  is a strictly decreasing and convex function in \(x\).
  We now compute the \(\alpha_x \in [0,1]\),
  where
  \begin{equation*}
    \frac{\partial}{\partial \alpha} \widehat{m}_{1/2}(x, \alpha_x)
    = \frac{\dif}{\dif x} \frac{1}{12 x + 12}
    = -\frac{1}{12 (x + 1)^2}
  \end{equation*}
  holds.
  This is equivalent to
  \begin{equation*}
    \log\frac{x+1}{x+\alpha_x+\tfrac{1}{2}}
    = -\frac{1}{12 (x + 1)^2} ,
  \end{equation*}
  which has the unique solution
  \begin{equation*}
    \alpha_x
    = (x+1) \exp\group[\bigg]{\frac{1}{12 (x+1)^2}} - (x+\tfrac{1}{2}) .
  \end{equation*}
  With this we choose the midpoint of \(\frac{1}{2}\) and \(\alpha_x\),
  \begin{equation*}
    \beta_x
    = \tfrac{1}{2}(\tfrac{1}{2} + \alpha_x)
    = \frac{x+1}{2} \exp\group[\bigg]{\frac{1}{12 (x+1)^2}} - \frac{x}{2} ,
  \end{equation*}
  to be where we will place the bound, that is we set
  \begin{equation*}
    \frac{1}{12 x + 12} = \frac{1}{12 (x + \beta_x) + K_x} ,
  \end{equation*}
  where
  \begin{equation*}
    K_x = 12 - 6\group[\bigg]{(x+1) \exp\group[\bigg]{\frac{1}{12 (x+1)^2}}  - x} .
  \end{equation*}
  With this we now have proven that
  \begin{align*}
    \widehat{m}_{1/2}(x, \alpha)
    \leq \widehat{m}_{1/2}(x, \tfrac{1}{2})
    &\leq \frac{1}{12 x + 12} \\
    &= \frac{1}{12 (x + \beta_x) + K_x}
    \leq \frac{1}{12 (x + \alpha) + K_x}
  \end{align*}
  holds for all \(\alpha \in {[0, \beta_x]}\).
  However,
  since
  \begin{equation*}
    \alpha \mapsto \widehat{m}_{1/2}(x, \alpha)
  \end{equation*}
  is a concave function for \(\alpha \in [\beta_x,1]\)
  and
  \begin{equation*}
    \alpha \mapsto \frac{1}{12 (x + \alpha) + K_x}
  \end{equation*}
  is a strictly decreasing and convex function for \(\alpha \in [\beta_x,1]\),
  it is straightforward to deduce that
  the tangent of
  \begin{equation*}
    \alpha \mapsto \widehat{m}_{1/2}(x, \alpha)
  \end{equation*}
  at \(\alpha_x\) will separate the two functions,
  showing that
  \begin{equation*}
    \widehat{m}_{1/2}(x, \alpha)
    \leq \frac{1}{12 (x + \alpha) + K_x}
  \end{equation*}
  also holds for all \(\alpha \in {[\beta_x, 1]}\).
  Finally,
  it is obvious that \(K_x\) is increasing in \(x\).
\end{proof}

From Lemma~\ref{lemma:goodburnsidestirlinggautschi}
we thus arrive at the following corollary:
\begin{corollary}\label{corollary:fairburnsidestirlinggautschi}
  Define \(\overline{b_{SG}^+} \colon \Rbbb_{\geq 0} \to \Rbbb\) by
  \begin{equation*}
    \overline{b_{SG}^+}(y) \isdef \frac{1}{12y + 12 - 6\exp(\tfrac{1}{12})} ,
  \end{equation*}
  and \(\overline{b_{SG}^\star} \colon \Rbbb_{\geq 0} \to \Rbbb\) by
  \begin{equation*}
    \overline{b_{SG}^\star}(y) \isdef \frac{1}{12y + 18 - 12 \exp(\tfrac{1}{48})} .
  \end{equation*}
  Then for all \((x,\alpha) \in \widehat{D}_{1/2}\) one has
  \begin{equation*}
    \widehat{m}_{1/2}(x, \alpha)
    \leq \overline{b_{SG}^+}(x + \alpha)
  \end{equation*}
  and,
  if \(x \geq 1\) holds additionally, also
  \begin{equation*}
    \widehat{m}_{1/2}(x, \alpha)
    \leq \overline{b_{SG}^\star}(x + \alpha) .
  \end{equation*}
\end{corollary}

From Corollary~\ref{corollary:fairburnsidestirlinggautschi}
we now finally arrive at the following
two-sided Stirling-Gautschi bounds of Burnside's type:
\begin{theorem}
  For all \((x, \alpha) \in \widehat{D}_{1/2}\) one has that
  \begin{align*}
    &\exp\group[\bigg]{-\frac{1}{24(x + \alpha) + 12}} \\
    &\qquad\leq \frac{\widehat{\Pi}(x, \alpha)}{S_{1/2}(x+\alpha)} \\
    &\quad\qquad\qquad \leq \exp\group[\bigg]{\frac{1}{12(x + \alpha) + 12 - 6\exp(\tfrac{1}{12})}} .
  \end{align*}
\end{theorem}

\begin{remark}\label{remark:nonasymptoticburnsidestirlinggautschi}
  Note that in Lemmas~\ref{lemma:burnsidestirlinggautschihalfbound}
  and~\ref{lemma:md-halfbound}
  we also gave results on the monotonicity
  from which one can also immediately arrive at non-asymptotic bounds:
  \begin{equation*}
    \sqrt{\frac{e}{\pi}}
    = \frac{\widehat{\Pi}(0, 0)}{S_{1/2}(0)}
    \leq \frac{\widehat{\Pi}(x, \alpha)}{S_{1/2}(x+\alpha)}
    \leq \frac{\widehat{\Pi}(0, \tfrac{1}{2})}{S_{1/2}(\tfrac{1}{2})}
    = \frac{e}{\sqrt{2\pi}}
  \end{equation*}
  holds for all \((x, \alpha) \in \widehat{D}_{1/2}\).
\end{remark}

As a special case of the preceeding theorem and remark,
we finally have the following corollary
for the logarithmically interpolated factorials:
\begin{corollary}
  For all \(x \in \Rbbb_{\geq 0}\)
  \begin{equation*}
    \sqrt{2\pi} \group[\bigg]{\frac{x + \tfrac{1}{2}}{e}}^{x + \frac{1}{2}} e^{-\frac{1}{24 x + 12}}
    \leq x\widehat{!}
    \leq \sqrt{2\pi} \group[\bigg]{\frac{x + \tfrac{1}{2}}{e}}^{x + \frac{1}{2}} e^{\frac{1}{12 x + 12 - 6\exp(1/12)}}
  \end{equation*}
  and
  \begin{equation*}
    \sqrt{2e} \group[\bigg]{\frac{x + \tfrac{1}{2}}{e}}^{x + \frac{1}{2}}
    \leq x\widehat{!}
    \leq e \group[\bigg]{\frac{x + \tfrac{1}{2}}{e}}^{x + \frac{1}{2}}
  \end{equation*}
  hold.
\end{corollary}

\section{Numerical Comparisons}\label{sec:numcomp}

We now illustrate the various bounds to indicate their effectiveness.
In Figure~\ref{figure:iotabounds} we show the behaviour of the bound
\begin{equation*}
  \widehat{\iota}(x, \alpha) = \log \frac{\widehat{\Pi}(x, \alpha)}{\Pi(x+\alpha)}
  \leq \overline{b_G^\star}(x + \alpha)
\end{equation*}
from Lemma~\ref{lemma:loginterpol-bound}.
This figure indicates that, while the bound is quite sharp
and has the correct asymptotic behaviour, it might be improved.
Indeed it seems likely that the following conjecture holds:
\begin{conjecture}
  Define \(\overline{b_G^\star} \colon \Rbbb_{\geq 0} \to \Rbbb\) by
  \begin{equation*}
    \overline{b_G^\star}(y) \isdef \frac{1}{8y + 4} ,
  \end{equation*}
  then for all \((x,\alpha) \in \widehat{D}\) one has
  \begin{equation*}
    \widehat{\iota}(x, \alpha) = \int_1^{\infty} \frac{\phi_\alpha(t)}{x+t} \dif\!t
    \leq \overline{b_G^\star}(x + \alpha) .
  \end{equation*}
\end{conjecture}

\begin{figure}[hb]
\centering
\includegraphics{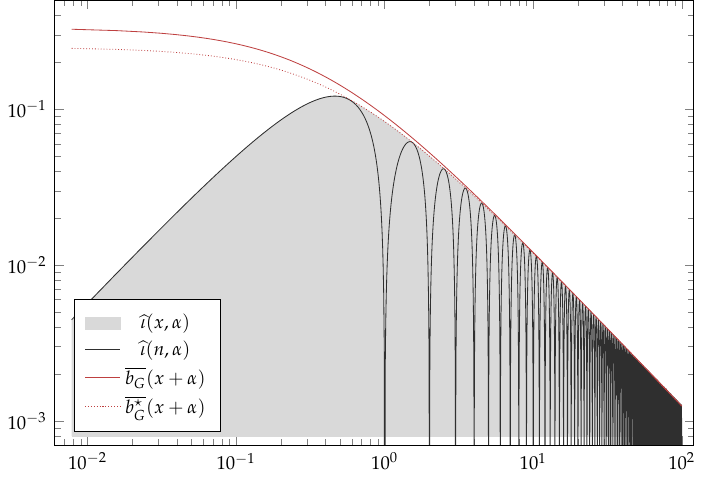}
\caption{Visualisation of the upper bounds for \(\widehat{\iota}(x, \alpha)\) in terms of \(x+\alpha\).}\label{figure:iotabounds}
\end{figure}

In Figure~\ref{figure:mbounds-half} the behaviour
of the bounds
\begin{equation*}
  \underline{b_S^{1/2}}(x)
  \leq m_{1/2}(x) = \log \frac{\Pi(x)}{S_{1/2}(x)}
  \leq \overline{b_S^{1/2}}(x)
\end{equation*}
from Lemma~\ref{lemma:md-halfbound} is shown.
This figure indicates that the bounds are quite sharp for all \(x \in \Rbbb_{\geq 0}\).
Similarily, Figure~\ref{figure:mbounds-zero} depicts the behaviour
of the bounds
\begin{equation*}
  \underline{b_S^{0}}(x)
  \leq m_{0}(x) = \log \frac{\Pi(x)}{S_{0}(x)}
  \leq \overline{b_S^{0}}(x)
\end{equation*}
from Lemma~\ref{lemma:md-zerobound}.
The figure clearly shows that the bounds are quite sharp for \(x \gg 0\),
but necessarily suffer at \(x = 0\) as \(S_{0}(0) = 0\) but \(\Pi(0) = 1\).
Lastly, the behaviour of the bounds
\begin{equation*}
  \underline{b_S^{1}}(x)
  \leq m_{1}(x) = \log \frac{\Pi(x)}{S_{1}(x)}
  \leq \overline{b_S^{1}}(x)
\end{equation*}
from Lemma~\ref{lemma:md-onebound} is shown in Figure~\ref{figure:mbounds-one}.
The figure clearly shows that the bounds are quite sharp for all \(x \in \Rbbb_{\geq 0}\).

\begin{figure}[phtb]
\centering
\includegraphics{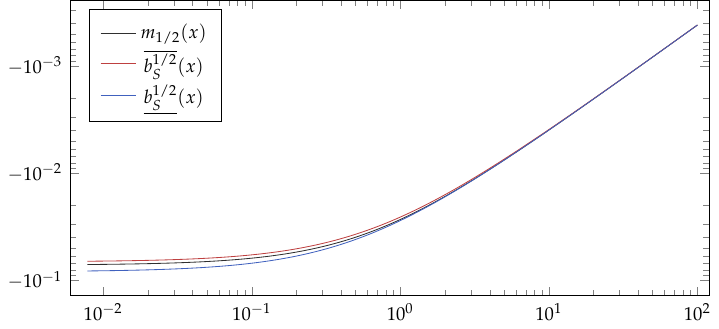}

    
\caption{Visualisation of the two-sided bounds for \(m_{1/2}(x)\) in terms of \(x\).}\label{figure:mbounds-half}
\end{figure}

\begin{figure}[phtb]
\centering
\includegraphics{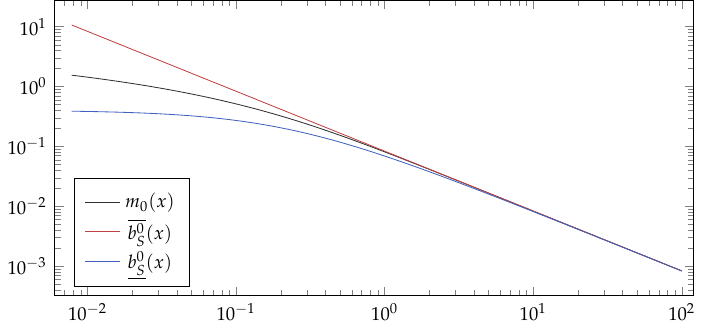}

    
\caption{Visualisation of the two-sided bounds for \(m_{0}(x)\) in terms of \(x\).}\label{figure:mbounds-zero}
\end{figure}

\begin{figure}[phtb]
\centering
\includegraphics{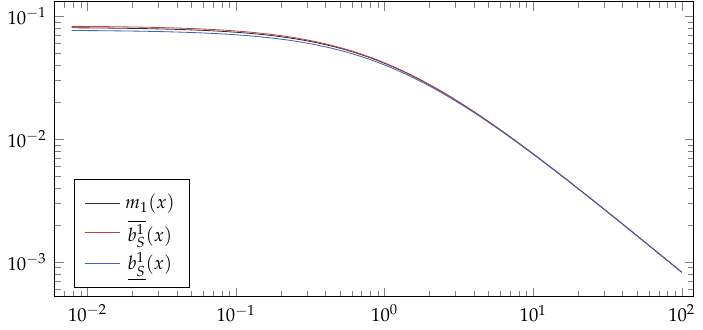}

    
\caption{Visualisation of the two-sided bounds for \(m_{1}(x)\) in terms of \(x\).}\label{figure:mbounds-one}
\end{figure}

Combining the bounds from Lemmas~\ref{lemma:md-halfbound} and~\ref{lemma:md-shiftbound}
one has the bounds
\begin{equation*}
  \underline{b_S^{1/2}}(x) + \underline{q_S^{-1/2}}(x)
  \leq m_{-1/2}(x) = \log \frac{\Pi(x)}{S_{-1/2}(x)}
  \leq \overline{b_S^{1/2}}(x) + \overline{q_S^{-1/2}}(x)
\end{equation*}
and
\begin{equation*}
  \underline{b_S^{1/2}}(x) + \underline{q_S^{2}}(x)
  \leq m_{2}(x) = \log \frac{\Pi(x)}{S_{2}(x)}
  \leq \overline{b_S^{1/2}}(x) + \overline{q_S^{2}}(x) .
\end{equation*}
The behaviour of these bounds are shown in Figures~\ref{figure:mbounds-neghalf}
and~\ref{figure:mbounds-two}.
It is clearly visible that while these bounds are not particularly sharp
for small values of \(x\),
they do show the correct asymptotic behaviour.

\begin{figure}[htb]
\centering
\includegraphics{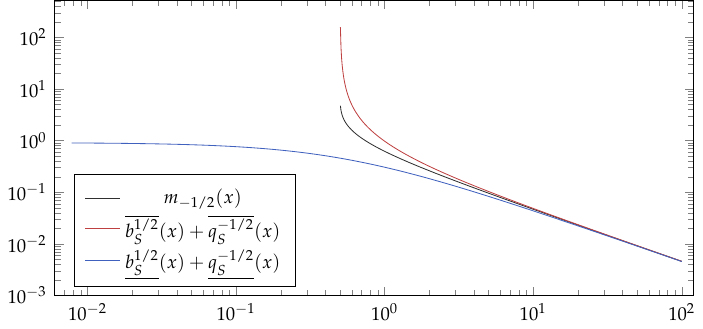}

    
\caption{Visualisation of the two-sided bounds for \(m_{-1/2}(x)\) in terms of \(x\).}\label{figure:mbounds-neghalf}
\end{figure}

\begin{figure}[htb]
\centering
\includegraphics{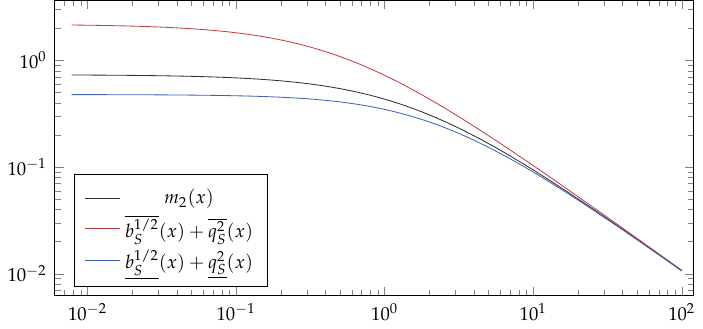}

    
\caption{Visualisation of the two-sided bounds for \(m_{2}(x)\) in terms of \(x\).}\label{figure:mbounds-two}
\end{figure}

Finally,
we also depict the bounds from Lemma~\ref{lemma:badburnsidestirlinggautschi}
and Corollary~\ref{corollary:fairburnsidestirlinggautschi} in Figures~\ref{figure:mhatbounds-half},
\ref{figure:mhatbounds-half2} and~\ref{figure:mhatbounds-vis}.
The figures clearly show that the bounds are quite sharp for all \(x \in \Rbbb_{\geq 0}\)
and that one might expect the following conjecture to hold:
\begin{conjecture}
  One has
  \begin{equation*}
    \widehat{m}_{1/2}(x, \alpha)
    \leq \overline{b_{SG}^\star}(x + \alpha)
  \end{equation*}
  for all \((x,\alpha) \in \widehat{D}_{1/2}\).
  In other words,
  the restriction \(x \geq 1\) in Corollary~\ref{corollary:fairburnsidestirlinggautschi}
  is not necessary.
\end{conjecture}

\begin{figure}[phtb]
\centering
\includegraphics{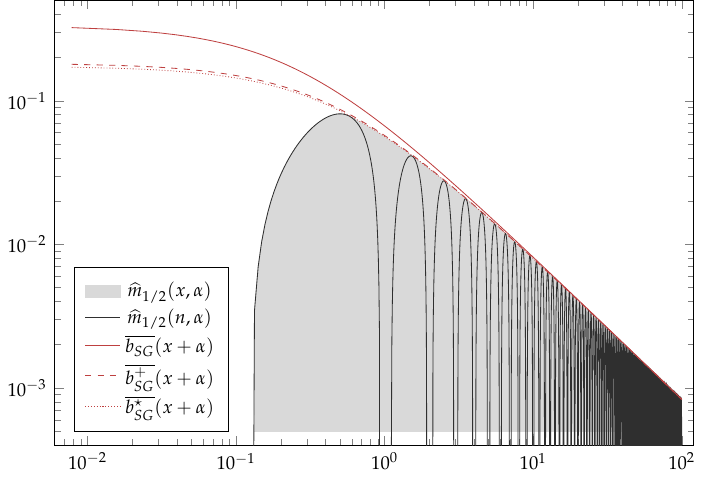}
\caption{Visualisation of the upper bounds for \(\widehat{m}_{1/2}(x, \alpha)\) in terms of \(x+\alpha\).}\label{figure:mhatbounds-half}
\end{figure}

\begin{figure}[phtb]
\centering
\includegraphics{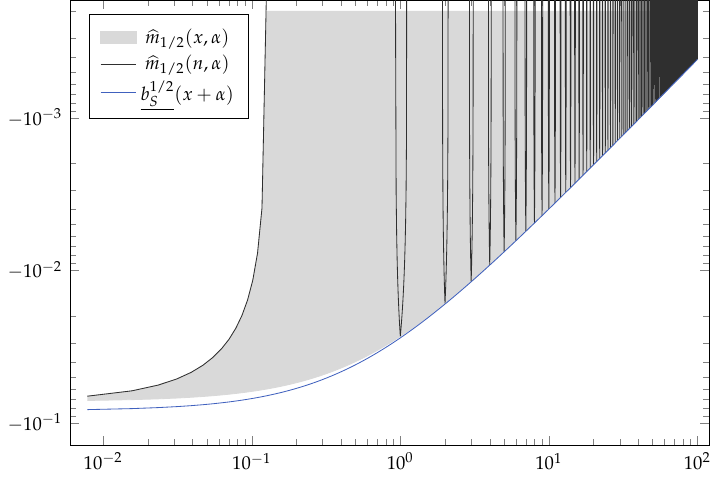}
\caption{Visualisation of the lower bounds for \(\widehat{m}_{1/2}(x, \alpha)\) in terms of \(x+\alpha\).}\label{figure:mhatbounds-half2}
\end{figure}

\begin{figure}[htb]
\centering
\includegraphics{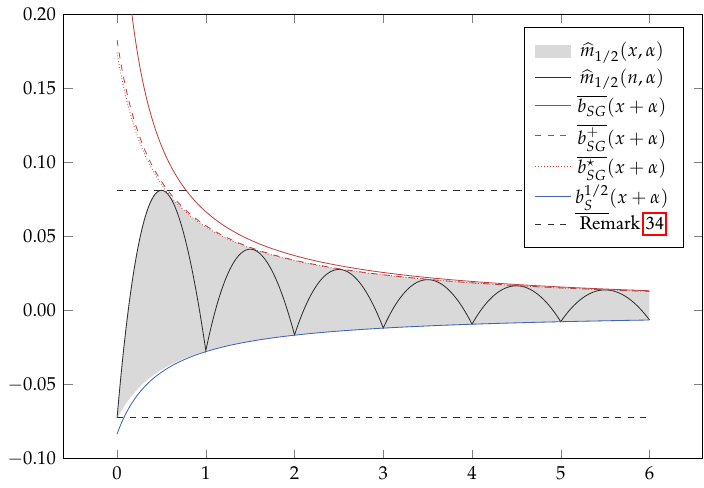}
\caption{Visualisation of the two-sided bounds for \(\widehat{m}_{1/2}(x, \alpha)\) in terms of \(x+\alpha\).}\label{figure:mhatbounds-vis}
\end{figure}

\section{Conclusion}

In this article, we have derived simple but reasonably effective two-sided bounds
for piecewise logarithmic interpolations of the pi function.
As far as reasonable we were able to use arguments in the proofs that are elementary,
such that the proofs themselves can also be considered simple.
Finally, we provided some numerical comparisons to depict the effectiveness of the bounds
and conjectured that two results might be improved further.

\section*{Acknowledgements}

The author wishes to thank Helmut Harbrecht and Christoph Schwab
for their input, which helped inspire parts of and shape this work.
The author also wishes to thank Carol Clarke for proof-reading the manuscript.

\bibliographystyle{plain}
\bibliography{refs}

\end{document}